\documentclass{article}
\usepackage{a4,graphics}
\usepackage[T1]{fontenc}
\usepackage{lmodern}
\usepackage[centertags]{amsmath}
\usepackage{amsfonts}
\usepackage{amssymb}
\usepackage{amsthm}
\usepackage{mathabx}
\usepackage{latexsym}
\usepackage[ansinew]{inputenc}
\usepackage[dvips]{graphicx,color}
\usepackage{psfrag}

\theoremstyle{plain}
\newtheorem{theorem}{Theorem}[section]
\newtheorem{proposition}[theorem]{Proposition}
\newtheorem{lemma}[theorem]{Lemma}
\newtheorem{corollary}[theorem]{Corollary}

\newcommand{\charf}[1]{\mathbf{1}_{#1}}
\newcommand{\real}{\mathbb R}
\newcommand{\nat}{\mathbb N}
\newcommand{\ent}{ \mathbb Z}

\newcommand{\dif}{\textrm{d}}

\newcommand{\norm}[1]{\left\lvert #1\right\rvert}
\newcommand{\dnorm}[1] {\left\lVert #1 \right\rVert}
\newcommand{\tnorm}[2] {\left\vvvert #1 \right\vvvert_{#2}}

\newcommand{\vect}[1]{\mathbf{#1}}
\newcommand{\var}[1]{\mathrm{Var}_{\raisebox{-2pt}{${\scriptstyle #1}$}}}
\newcommand{\proj}[1] {\Pi_{#1}}

\newcommand{\hinv}[1]{h_{\mathrm{inv}}^{\scriptscriptstyle (\!#1\!)}}
\newcommand{\muinv}[1]{\mu_{\mathrm{inv}}^{\scriptscriptstyle (\!#1\!)}}
\newcommand{\sleb}[1]{m^{\scriptscriptstyle (\!#1\!)}}

\begin{document}
\title{Phase Transition and Correlation Decay\\ in Coupled Map Lattices }
\author{A. de Maere\footnote{Partially supported by the Belgian IAP program P6/02.}\\UCL, FYMA, Chemin du Cyclotron 2,\\B-1348 Louvain-la-Neuve, Belgium\\ augustin.demaere@uclouvain.be}

\date{}
\maketitle

\begin{abstract}
For a Coupled Map Lattice with a specific strong coupling emulating Stavskaya's probabilistic cellular automata, we prove the existence of a phase transition using a Peierls argument, and exponential convergence to the invariant measures for a wide class of initial states using a technique of decoupling originally developed for weak coupling. This implies the exponential decay, in space and in time, of the correlation functions of the invariant measures.
\end{abstract}

\section{Introduction}
\label{intro}

It is now well-known that infinite dimensional systems are radically different from their  finite dimensional counterparts, and perhaps the most striking difference is the phenomenon of phase transition. In general, finite dimensional systems tend to have only one natural measure, also called phase. For infinite dimensional systems, the picture is quite different: weakly coupled systems tend to have only one natural measure and strongly coupled systems may have several.

This picture also holds for Coupled Map Lattices (CML). CML are discrete time dynamical systems generated by the iterations of a map on a countable product of compact spaces. The map is the composition of a local dynamic with strong chaotic properties and a coupling which introduce some interaction between the sites of the lattice. CML were introduced by Kaneko \cite{Ka84,Ka93}, and they can be seen as an infinite dimensional generalization of interval maps. Their natural measures are the SRB measures and in this case, the definition of SRB measure is a measure invariant under the dynamic with finite dimensional marginals of bounded variation. The unicity of the SRB measure for weakly Coupled Map Lattices has been thoroughly studied in various publications \cite{BuSi88,PeSi91,BrKu95,BrKu96,BrKu97,MaVa97,BaDeIsJaKu98,JiPe98,Ke98,JaJa01,KeLi04,KeLi05,Ja05,KeLi06}.

Despite many numerical results on the existence of phase transition for strongly coupled map lattices (see for instance \cite{MiHu93,BoBuCoFrPe95,BoBuCoFrPe01,ScJu01}), there are still few analytical results on the subject. The first rigorous proof of the existence of a phase transition was performed by Gielis and MacKay \cite{GiMa00}, who constructed a bijection between some Coupled Map Lattices and Probabilistic Cellular Automata (PCA) and relied on the existence of a phase transition for the PCA to prove the existence of a phase transition for the CML. But their result requires the assumption that the coupling does not destroy the Markov partition of the single site dynamics, and this hypothesis is clearly not true for general Coupled Map Lattices. Other publications are following this approach by considering specific coupling that preserve the Markov partition \cite{Ju01,BlBu03}. Later, Bardet and Keller \cite{BaKe06} proved the existence of a phase transition for a more natural coupled map lattice emulating Toom's probabilistic cellular automata, using a standard Peierls argument.

The purpose of this article is to extend these results for a Coupled Map Lattice with a very general local dynamic and a coupling behaving like Stavskaya's PCA.

\section{Description of the Model and Main Results}\label{descr}

\subsection{General setup}

Let $I = [-1,1]$, and $X = I^{\ent}$. The Coupled Map Lattice is given by a map ${T : X \to X}$, where $T = \Phi_{\epsilon} \circ \tau^{\ent}$ with $\tau : I \to I$ the local dynamic and $\Phi_{\epsilon}: X \to X$ the coupling. The evolution of initial signed Borel measures under the dynamic is given by the transfer operator $T$, also called the Perron-Frobenius operator, which is defined by:
\begin{equation*}
T \mu (\varphi) = \mu(\varphi \circ T). 
\end{equation*}
Let $m_{\ent}$ be the Lebesgue measure on $X$. Let $\mathcal{C}(X)$ be the set of continuous real-valued functions on $X$, and $\norm{\, \cdot\, }_{\infty}$ be the sup norm on this space.

For every finite $\Lambda \subset \ent$, let $\norm{\Lambda}$ be the cardinality of $\Lambda$, $\pi_{\Lambda}: X \mapsto I^{\Lambda}$ be the canonical projector from $X$ to $I^{\Lambda}$, $m_\Lambda$ the Lebesgue measure on $I^{\Lambda}$, and $\pi_{\Lambda}\mu$ the restriction of $\mu$ to $I^{\Lambda}$. Then, for every signed Borel measure $\mu$, the total variation norm is defined by:
\begin{equation}\label{eq1}
\norm{\mu} = \sup\big\{\,\mu(\varphi)\,\big|\,\varphi \in \mathcal{C}(X) \,\textrm{and}\, \norm{\varphi}_{\infty} \leq 1\,\big\}.
\end{equation}
Consider $L^1(X)$, the space of signed Borel measures such that $\norm{\mu} < \infty$ and $\pi_{\Lambda}\mu$ is absolutely continuous with respect to $m_{\Lambda}$ for every finite $\Lambda \subset \ent$. We immediately see that if the map $T$ is piecewise continuous, Proposition \ref{propA2} from the Appendix implies that:
\begin{equation}\label{eq1.1}
\norm{T \mu} \leq \norm{\mu}.
\end{equation}
Note that if $\mu$ is a probability measure, its total variation norm is always equal to $1$.

It is well-known that the total variation norm is not sufficient to study the spectral properties of Coupled Map Lattices \cite{JaJa01}, and that the bounded variation norm also plays an important role. Let $\dnorm{\,\cdot\,}$ be the bounded variation norm, defined by:
\begin{equation}\label{eq2}
\dnorm{\mu} = \sup\big\{\,\mu(\partial_p\varphi)\,\big|\,p \in \ent\,,\,\varphi \in \mathcal{C}^{1}(X) \,\textrm{and}\, \norm{\varphi}_{\infty} \leq 1\,\big\}.
\end{equation}
It can be seen that the space ${\mathcal{B}(X) = \big\{\mu \in L^{1}(X)\; \big|\; \dnorm{\mu} < \infty \big\}}$, endowed with the norm $\dnorm{\,\cdot\,}$ is a Banach space. If we use the fact that for any continuous function $\varphi$, we have $\varphi(\vect{x}) = \partial_p \int_{0}^{x_p} \varphi(\xi_p,\vect{x}_{\neq p})\,\dif \xi_p$, we can also prove that:
\begin{equation}\label{eq2.1}
\norm{\mu} \leq \dnorm{\mu}.
\end{equation}

Following an original idea of Vitali \cite{Vi08}, we also consider: 
\begin{equation}\label{eq7.0}
\var{\Lambda}\mu = \sup \Big\{\,\mu\big( \partial_{\Lambda}\varphi\;\big) \ \Big|\ \varphi \in \mathcal{C}^1_{\Lambda}(X)\; \mathrm{and}\; \norm{\varphi}_{\infty} \leq 1\, \Big\}
\end{equation}
for any finite $\Lambda \subset \ent$, where $\partial_{\Lambda}$ denotes the derivative with respect to all the coordinates in $\Lambda$ and $\mathcal{C}^1_{\Lambda}(X)$ is the set of continuous functions $\varphi$ such that $\partial_\Lambda \varphi$ is also continuous. We already note that:
\begin{equation*}
\begin{cases}\var{\varnothing}\mu = \norm{\mu} \\
 \sup_{p \in \ent} \var{\{p\}}\mu = \dnorm{\mu}
\end{cases}
\end{equation*}
In general, we do not expect the variation $\var{\Lambda}\mu$ to be bounded uniformly in $\Lambda$. In fact, even for a totally decoupled measure of bounded variation $\mu$, it is straightforward to check that $\var{\Lambda}\mu$ will grow exponentially with $\norm{\Lambda}$. Consequently, it is natural to consider the following $\theta$-norm, for some $\theta > 1$:
\begin{equation}\label{eq9}
\tnorm{\mu}{\theta} = \sup_{\Lambda \subset \ent}\  \theta^{- \norm{\Lambda}}\,  \var{\Lambda}\mu .
\end{equation}

For any $K >0$, $\alpha >0$ and $\theta \geq 1$, let $\mathcal{B}(K,\alpha,\theta)$ be the set of measures in $\mathcal{B}(X)$ such that for all finite $\Lambda \subset \ent$, we have,
\begin{equation}\label{eq21}
\tnorm{\charf{(0,1]^{\Lambda}}\mu}{\theta} \leq K {\alpha}^{\norm{\Lambda}},
\end{equation}
where $(0,1]^\Lambda \subseteq X$ is the set of configurations $\vect{x}$ such that $x_p \in (0,1]$ for every $p \in \Lambda$ and for any $A \subseteq X$, $\charf{A}$ is used as an operator acting on measures through:
\begin{equation*}
\charf{A}\mu(\varphi) = \mu \left(\charf{A}\varphi\right).
\end{equation*}
 
Let us just give an example of some measure in $\mathcal{B}(K,\alpha,\theta)$. If $\norm{\cdot}_{L^1(I)}$ and $\dnorm{\cdot}_{BV}$ are respectively the total variation norm and bounded variation norm on functions, if $h^{\scriptscriptstyle (\!-\!)}$ and $h^{\scriptscriptstyle (\!+\!)}$ are two probability densities of bounded variation on $[-1,0]$ and $(0,1]$ respectively and if $\mu = \prod_{p \in \ent} h(x_p)\,\dif x_p$ with  $h = \alpha h^{\scriptscriptstyle (\!+\!)} + (1-\alpha) h^{\scriptscriptstyle (\!-\!)}$  for some $\alpha \in [0,1]$, we can check that:
\begin{equation*}
\var{\Omega} \charf{(0,1]^{\Lambda}}\mu \leq \alpha^{\norm{\Lambda}}\,\dnorm{h^{\scriptscriptstyle (\!+\!)} }_{BV}^{\norm{\Omega \cap \Lambda}} \,\dnorm{h}_{BV}^{\norm{\Omega\setminus \Lambda}}.
\end{equation*}
Hence, as long as $\theta > \max \{\dnorm{h^{\scriptscriptstyle (\!+\!)} }_{BV},\dnorm{h}_{BV}\}$, we have $\tnorm{\charf{(0,1]^{\Lambda}}\mu}{\theta} \leq \alpha^{\norm{\Lambda}}$ and so $\mu$ belongs to $\mathcal{B}(1,\alpha,\theta)$.

\subsection{Assumptions on the dynamic}

We will assume the following properties of the dynamic. The coupling ${\Phi_{\epsilon}: X \mapsto X}$ depends on some parameter $\epsilon \in [0,1]$ and is explicitly given by:
\begin{equation*}
\Phi_{\epsilon}(\vect{x})_p = \left\{\begin{array}{ll}
x_p & \mathrm{if}\; x_p > 0\; \mathrm{and}\; x_{p+1} > 0  \\
x_p - 1 + \epsilon \rule{10pt}{0pt}& \mathrm{if}\; x_p > 0 \;\mathrm{and}\; x_{p+1} \leq 0   \\
x_p + \epsilon & \mathrm{if}\; x_p \leq 0  \\
\end{array}\right.
\end{equation*}
The coupling $\Phi_{\epsilon}$ has a behavior similar to Stavskaya's probabilistic cellular automata (see \cite{St73,DoKrTo90} for more details on Stavskaya's PCA). Indeed, if both $x_p$ and $x_{p+1}$ are strictly positive, $x_p$ will be sent to the interval $(0,1]$, and if $x_p$ or $x_{p+1}$ are negative, $x_p$ will be sent on $[-1,0]$, except if $x_p$ is in the small subset $(-\epsilon,0]\cup(1-\epsilon,1]$. If $\epsilon$ is close to $0$, the system is strongly coupled, and if $\epsilon$ is close to $1$, the system is weakly coupled.

On the other hand, we will assume that the single site dynamics $\tau$ is a piecewise expanding map $\tau: I \mapsto I$ such that:
\begin{itemize}
\item  $\exists \, \zeta_1, \ldots \zeta_{N+1} \in [-1,1]$, where $-1 = \zeta_1 < \ldots < \zeta_{N+1} = 1$ and $J_i =  (\zeta_i, \zeta_{i+1})$, such that the restriction of $\tau$ to the interval $J_i$ is monotone and uniformly $\mathcal{C}^{2}(I)$.
\item $\kappa = \inf \norm{\tau'} > 2$ and there is some $D_0 > 0$ such that $\frac{2}{\kappa \min \norm{J_i}} +  \lvert\frac{\tau''}{\left(\tau'\right)^2}\rvert_{\infty} \leq D_0$.
\item The map $\tau$ has two non-trivial invariant subsets $[0,1]$ and $[-1,0]$, and the dynamic restricted to these subsets is mixing. For the sake of simplicity, we will assume the map on $[-1,0]$ to be the translation of the map on $(0,1]$: for every $x \in (0,1]$, $\tau(x-1) = \tau(x) -1$.
\end{itemize}

If $P_{\tau}$ is the Perron-Frobenius operator associated to $\tau$ and $\lambda_0 = \frac{2}{\kappa}$, these assumptions imply the Lasota-Yorke inequality \cite{LaYo73}:
\begin{equation}
\dnorm{ P_{\tau}^{m} h}_{BV}  \leq \lambda_0^m \dnorm{h}_{BV} + \tfrac{D_0}{1-\lambda_0} \norm{h}_{L^1(I)},\label{eq0.1}
\end{equation}
and this inequality puts strong constraints on the spectrum of $P_{\tau}$ as an operator acting on functions of bounded variation. Indeed, the Ionescu Tulcea-Marinescu theorem \cite{IoMa50,He93} shows us that the spectrum of $P_{\tau}$ in the space of functions of bounded variation consists of the doubly degenerate eigenvalue $1$ with the rest of the spectrum contained inside a circle of radius $\lambda_0$.

Since $[-1,0]$ and $(0,1]$ are invariant subsets, we know that we can choose the two invariant densities associated to the eigenvalue $1$ to be respectively concentrated on $[-1,0]$ and $(0,1]$. Let $\hinv{-}$ and $\hinv{+}$ be these two eigenvectors. The Gelfand formula implies then that we can always choose $\varsigma \in (0,1)$ with $\varsigma > \lambda_0$ and $c > 0$ such that, for any function $h$ on $[-1,0]$ of bounded variation and any $m \in \nat$, we have:
\begin{equation}
\norm{P_{\tau}^{m} h  - \left( \int h(x) \  \mathrm d x\right)\;\hinv{-}}_{L^1(I)}  \leq c\  {\varsigma}^m  \dnorm{h}_{BV}.\label{eq0.2}
\end{equation}
A similar result also holds for $\hinv{-}$ and any function $h$ on $(0,1]$ of bounded variation.

Let $\mathcal{T}(D_0,c,\varsigma)$ be the set of maps $\tau$ satisfying the above assumptions for given values of $D_0 > 0$, $c > 0$ and $\varsigma \in (0,1)$ and arbitrary values of $\lambda_0 \in (0,\varsigma)$. We can see for instance that the Bernoulli shift or the maps introduced in \cite{Li95,Ke99} extended on the interval $[-1,1]$ using the symmetry assumptions all belong to one of the $\mathcal{T}(D_0,c,\varsigma)$. It can be seen that, if some map $\tau$ belongs to $\mathcal{T}(D_0,c,\varsigma)$, then $\tau^2 = \tau \circ \tau$ also belongs to $\mathcal{T}(D_0,c,\varsigma)$. And since $\inf \norm{(\tau^2)'} \geq (\inf \norm{\tau'})^2$, this implies that if $\mathcal{T}(D_0,c,\varsigma)$ is not empty, it contains maps with arbitrary large values of $\kappa$. This point will become important later.

Let $E_{\epsilon} = \tau^{-1}(-\epsilon,0]\cup\tau^{-1}(1-\epsilon,1]$. The assumptions on $\tau$ imply that the Lebesgue measure of $E_{\epsilon}$ is of order at most $\epsilon$. Indeed, since $\norm{\tau'}$ is bounded from below and since $\tau$ preserves the intervals $[0,1]$ and $[-1,0]$, we know that the preimage of ${(-\epsilon,0]\cup(1-\epsilon,1]}$ under $\tau$ consists of intervals of length at most $\frac{\epsilon}{\kappa}$, and there are at most $N$ such intervals. One could be worried about the fact that $N$ seems to be unbounded in the assumptions on $\tau$, but this is not the case, because $N \leq \frac{2}{\min_{i} \norm{J_i}}$ and so, $N \leq \kappa D_0$. Therefore, we have:
\begin{equation}\label{eq7.1}
\norm{E_{\epsilon}} \leq \frac{N}{\kappa}\,\epsilon \leq D_0 \, \epsilon.  
\end{equation}
For the commodity, we also introduce  the following constants:
\begin{align}
\lambda_1 & = \frac{4}{\kappa} + \frac{D_0 \norm{E_{\epsilon}}}{2} & D_1   &= \frac{4}{\kappa \min_i \norm{J_i}}\label{eq3.2}
\end{align}

\subsection{Main results}

We immediately see that for any value of $\epsilon$, the measure $\muinv{+}$ defined as:
\begin{equation*}
\muinv{+} = \prod_{p \in \ent} \hinv{+}(x_p)\,\dif x_p 
\end{equation*}
is always invariant under $T$. If $\epsilon$ is close to $1$, we can consider the system as a small perturbation of the case $\epsilon = 1$, and use a simple modification of the decoupling technique introduced by Keller and Liverani \cite{KeLi06} to prove that $\muinv{+}$ is indeed the unique SRB measure. Since $\muinv{+}$ is totally decoupled, it trivially has the property of exponential decay of correlation in space. Furthermore, as a direct consequence of the decay of correlation in time for the single-site dynamics (which comes from (\ref{eq0.2}), see \cite{Ba00} for more details.), we also have the decay of correlations in time for $\muinv{+}$.

We will prove in Section \ref{peierls} that if we decrease the strength of the coupling, other SRB measures may appear, and the system therefore undergoes a phase transition. For this, let us first define $\alpha_0$ as:
\begin{equation}\label{eq3.3}
\alpha_0 =  \frac{1}{2} \ \max\left\{\lambda_1 + \sqrt{\lambda_1^2 + 2 D_1 \norm{E_{\epsilon}}}, \,\frac{D_0 \norm{E_{\epsilon}}}{1 - \lambda_0}\,\right\} .
\end{equation}
Since $\norm{E_{\epsilon}} \leq D_0\, \epsilon$, we have:
\begin{equation}\label{eq21.-1}
\lim_{\epsilon \to 0} \,\alpha_0 = \lim_{\epsilon \to 0} \lambda_1 = \lim_{\epsilon \to 0} \left(\frac{4}{\kappa} + \frac{D_0 \norm{E_{\epsilon}}}{2}\right) = \frac{4}{\kappa}.
\end{equation}
Then, the existence of a phase transition is a consequence of this Theorem:

\begin{theorem}[Existence of a phase transition]\label{result1}Assume that $\tau$ belongs to $\mathcal{T}(D_0,c,\varsigma)$ and that $\kappa > 108$. Then, there is some $\epsilon_0 > 0$ such that, if $\epsilon \in [0, \epsilon_0)$:
\begin{itemize}
\item[\textbullet] the dynamic $T$ admits another SRB measure $\muinv{-} \neq \muinv{+}$
\item[\textbullet] \rule{0pt}{14pt}$\muinv{-}$ belongs to $\mathcal{B}(K_0, 3 \alpha_0,\theta_0)$ with $\alpha_0 < \frac{1}{27}$ defined in (\ref{eq3.3}), and $K_0$ and  $\theta_0$ given by:
\begin{equation}\label{eq3.3bis}\theta_0  = \frac{2 \alpha_0}{\norm{E_{\epsilon}}}  \qquad  \qquad K_0 =  \frac{1}{2 (1 - 27 \alpha_0)(1 - 3 \alpha_0)}
\end{equation}
\end{itemize}
 
\end{theorem}

The strategy used in the proof of this result is similar to the one used by Bardet and Keller in \cite{BaKe06} in the sense that is also use a Peierls argument, but the contour estimates are done in a different way, giving us a stronger result which allows us to prove in Section \ref{expconv} that a wide class of initial measures converges exponentially fast towards $\muinv{-}$.
\begin{theorem}[Exponential convergence to equilibrium]\label{result2}
Assume that $\tau$ belongs to $\mathcal{T}(D_0,c,\varsigma)$ and that $\kappa$ is larger than some $\kappa_1$ that depends on $D_0$, $c$ and $\varsigma$. Then, there is some $\epsilon_1 \in (0,1)$ such that, if $\epsilon \in [0,\epsilon_1]$, there is some $\sigma < 1$ such that for any $K >0$ there is some constant $C >0$ such that:
\begin{equation*}
\norm{\,\mu(\,\varphi \circ T^t\,) - \muinv{-}(\,\varphi\,)\,} \leq C\, \norm{\Lambda}\ \sigma^t\,\norm{\varphi}_{\infty}
\end{equation*}
for any probability measure $\mu$ in $\mathcal{B}(K,3 \alpha_0,\theta_0)$ and for any continuous function $\varphi$ depending only on the variables in $\Lambda \subset \ent$.
\end{theorem}

Eventually, we will show in Section \ref{expdecay} that Theorem \ref{result2} implies the exponential decay of correlations for the invariant measure $\muinv{-}$ both in space and time, therefore showing that $\muinv{-}$ is not only a SRB measure, but an extremal one:
\begin{proposition}[Exponential decay of correlations in space]\label{result3}
Under the assumptions of Theorem \ref{result2}, there is some positive constant $C$ such that for any bounded continuous functions $\varphi$ and $\psi$ depending on finitely many variables, respectively in $\Lambda$ and $\Omega$, we have:
\begin{align*}
\norm{\,\muinv{-}( \varphi \, \psi ) - \muinv{-}(\varphi)\ \muinv{-}(\psi)\,} \leq C\ \norm{\Lambda \cup \Omega}\ \sigma^{d(\Lambda,\Omega)}\ \norm{\varphi}_{\infty}\, \norm{\psi}_{\infty}.
\end{align*}
where $d(\Lambda,\Omega) = \inf \{\,\norm{q -p}\;|\; q \in \Lambda, \, p \in \Omega\,\}$ is the distance between $\Lambda$ and $\Omega$.
\end{proposition}

\begin{proposition}[Exponential decay of correlations in time]\label{result4}
Under the assumptions of Theorem \ref{result2}, for any functions $\varphi$ and $\psi$ depending only on the variables in some finite $\Lambda \subset \ent$, with $\varphi \in \mathcal{C}(X)$ and $\psi \in \mathcal{C}_{\Lambda}^1 (X)$, there is some constant $C_{\varphi,\psi} > 0$ such that:
\begin{align*}
\norm{\,\muinv{-}( \varphi\circ T^t\, \psi ) - \muinv{-}(\varphi)\ \muinv{-}(\psi)\,} \leq C_{\varphi,\psi}\ \norm{\Lambda} \sigma^{t}.
\end{align*}
\end{proposition}

\section{Existence of a phase transition}
\label{peierls}

\subsection{Cluster expansion} 

For $n \in \nat$ fixed and for any finite $\Lambda \subset \ent$, let ${\mathcal{E}(\Lambda) \subseteq X}$ be the set of configurations $\vect{x}$ such that $T^n x_{p} > 0$ for every $p \in \Lambda$, where $T^n x_{p}$ is a notation for $(T^n \vect{x})_{p}$. Then, to any $\vect{x}$ in $\mathcal{E}(\Lambda)$, we can associate a cluster ${\Gamma \subseteq \ent\times \{0,\ldots,n\}}$ using the following rules:
\begin{enumerate}
\item At time $n$, we add every $(p,n)$ with $p \in \Lambda$ to $\Gamma$.
\item For every $t \in \{0,\ldots,n-1\}$ and starting from $t = n-1$, if $(p,t+1)$ already belongs to $\Gamma$, $T^t x_p > 0$ and $T^t x_{p+1} > 0$, we add $(p,t)$ and $(p+1,t)$ to $\Gamma$.
\end{enumerate}
An example of such a cluster can be found in Figure \ref{fig1}.
\begin{figure}
\centering
\includegraphics{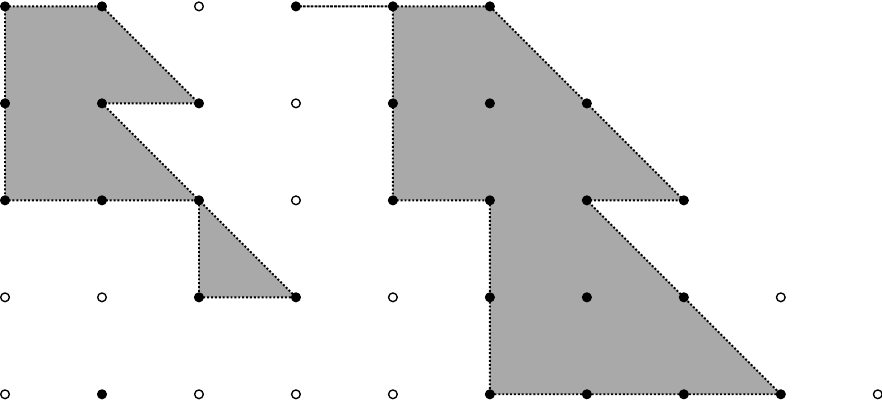}
\caption{Example of a cluster. White points are negative sites and black points positive sites.\label{fig1}}
\end{figure}

Let $g$ be the application mapping $\vect{x}$ onto $\Gamma$, and $\mathcal{G}(\Lambda)$ be the image of $\mathcal{E}(\Lambda)$ under $g$. Then:
\begin{equation}
\mathcal{E}(\Lambda) = \bigcup_{\Gamma \in \mathcal{G}(\Lambda)} g^{-1}\Gamma,
\end{equation}
or equivalently, in term of characteristic functions:
\begin{equation}
\charf{(0,1]^\Lambda}(T^n \vect{x}) = \sum_{\Gamma \in \mathcal{G}(\Lambda)}\charf{g^{-1}\Gamma}(\vect{x})\label{eq4.-1}
\end{equation}
If we define  $\partial \Gamma$ by
\begin{equation}
\partial \Gamma = \Big\{\,(p,t) \in \Gamma \, \Big|\, t = 0\, \mathrm{or}\, (p,t-1) \notin \Gamma\, \mathrm{or}\, (p+1,t-1) \notin \Gamma \,\Big\},
\end{equation}
and if we define $\Gamma_t = \{\, q \in \ent \,|\, (q,t) \in \Gamma\,\}$ and $\partial \Gamma_t = \{\, q \in \ent \,|\, (q,t) \in \partial \Gamma\,\}$ the restrictions of respectively $\Gamma$ and $\partial \Gamma$ to time $t$, we can see that the characteristic function of $g^{-1}\Gamma$ can be rewritten as:
\begin{align}
& \charf{g^{-1} \Gamma} \left(\, \vect{x}\,\right) \nonumber\\
&\qquad = \prod_{t = 0}^{n}\ \bigg[ \prod_{p \in \Gamma_t } \charf{(0,1]}(T^t x_p)  \prod_{p \in \partial\Gamma_{t+1}} \left(1 - \charf{(0,1]}(T^t x_p) \ \charf{(0,1]}(T^t x_{p+1})\right) \bigg]\nonumber\\
& \qquad = \prod_{t = 0}^{n}\ \charf{E(\Gamma,t)}(T^t  \vect{x}).\label{eq4}
\end{align}
where $E(\Gamma,t) \subseteq X$ is defined by:
\begin{align}
& \charf{E(\Gamma,t)} \left(\vect{x}\right)\nonumber\\
& \qquad = \prod_{p \in \Gamma_t } \charf{(0,1]}(x_p)  \prod_{p \in \partial\Gamma_{t+1}} \left(1 - \charf{(0,1]}(x_p) \ \charf{(0,1]}(x_{p+1})\right).\label{eq5}
\end{align}

Let us pick some arbitrary $\Gamma \in \mathcal{G}(\Lambda)$ . The cluster $\Gamma$ can be splitted in connected parts, respectively $\Gamma^{\scriptscriptstyle (k)}$ for $k = 1,\ldots,c$ with $c$ the number of connected parts. For any connected part of $\Gamma$, let say $\Gamma^{\scriptscriptstyle (k)}$, we define $\Lambda^{\scriptscriptstyle(k)} = \{p\,|\,(p,n) \in \Gamma^{\scriptscriptstyle(k)}\}$. The outer boundary of $\Gamma^{\scriptscriptstyle(k)}$ is now a closed loop, and we can always choose the orientation of the loop to be clockwise. The outer path of $\Gamma^{\scriptscriptstyle(k)}$ is now defined as the part of the closed loop that goes from $(\sup \Lambda^{\scriptscriptstyle(k)},n)$ to $(\inf \Lambda^{\scriptscriptstyle(k)},n)$. The outer paths associated to the cluster of Figure \ref{fig1} have been drawn at Figure \ref{fig2}.
\begin{figure}
\centering
\includegraphics{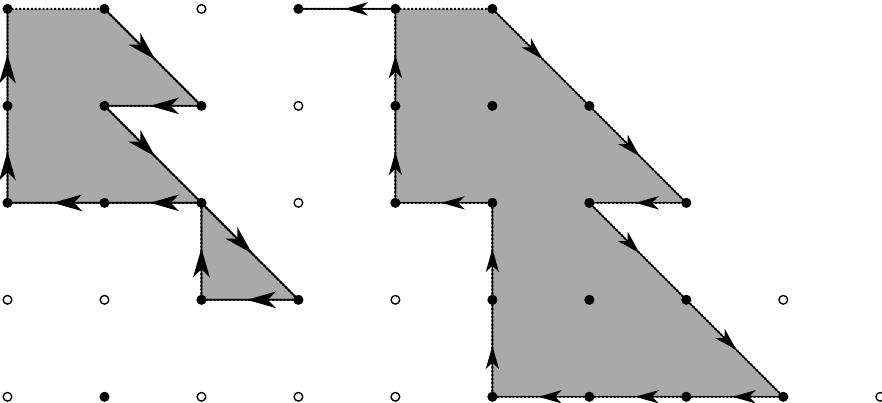}
\caption{Outer paths associated to the cluster of Figure \ref{fig1}.\label{fig2}}
\end{figure}

One can see that the cluster $\Gamma^{\scriptscriptstyle(k)}$ is univoquely defined by its outer path and that the outer path only makes jumps along the edges $(+1,-1)$, $(-1,0)$ and $(0,+1)$. Let $n_d^{\scriptscriptstyle(k)}$, $n_v^{\scriptscriptstyle(k)}$ and $n_h^{\scriptscriptstyle(k)}$ be the number of jumps in these directions respectively, and let $\partial \Gamma^{\scriptscriptstyle(k)} = \partial \Gamma \cap \Gamma^{\scriptscriptstyle(k)}$. Then, since the outer path starts in $(\sup \Lambda^{\scriptscriptstyle(k)},n)$ and ends in $(\inf \Lambda^{\scriptscriptstyle(k)},n)$, and since there is always an horizontal edge between two sites of the outer path belonging to $\partial \Gamma^{\scriptscriptstyle(k)}$, we have: 
\begin{equation*}
\begin{cases}
n_h^{\scriptscriptstyle(k)} \geq n_d^{\scriptscriptstyle(k)}  + \norm{\Lambda^{\scriptscriptstyle(k)}} -1\\
n_d^{\scriptscriptstyle(k)} - n_v^{\scriptscriptstyle(k)}= 0\\
\norm{\partial \Gamma^{\scriptscriptstyle(k)}} \geq  n_h^{\scriptscriptstyle(k)}+1
\end{cases}
\end{equation*}

We can now go back to the cluster $\Gamma$ by summing over the $c$ connected parts and defining $n_d = \sum_{k = 1}^c n_d^{\scriptscriptstyle(k)}$, $n_v = \sum_{k = 1}^c  n_v^{\scriptscriptstyle(k)}$ and $n_h = \sum_{k = 1}^c  n_h^{\scriptscriptstyle(k)}$ yields:
\begin{equation}\begin{cases}
{\displaystyle n_h \geq n_d + \norm{\Lambda} -c}\\
n_d =   n_v\\
\norm{\partial \Gamma} \geq  n_h +c
\end{cases}\label{eq7}\end{equation}

If we want to estimate the probability with respect to some initial signed measure $\mu$ that at time $n$ all the sites in $\Lambda$ are positive, we can now use (\ref{eq4.-1}) and (\ref{eq4}):
\begin{equation}\label{eq6}
\,\charf{(0,1]^{\Lambda}}\,T^n \mu = \sum_{\Gamma \in \mathcal{G}(\Lambda)}\,T^n\, \charf{g^{-1}\Gamma} \mu = \sum_{\Gamma \in \mathcal{G}(\Lambda)} \,\charf{(0,1]^{\Lambda}}\,\Bigl[\prod_{t = 0}^{n-1} \left(\, T \charf{E(\Gamma,t)}\,\right)\Bigr]  \mu.
\end{equation}
Of course, we assumed that the product of operators $\prod_{t = 0}^{n-1} \left(\, T \charf{E(\Gamma,t)}\right)$ is time-ordered.

The expansion of equation (\ref{eq6}) can be the starting point of what is called in Statistical Mechanics a Peierls argument: indeed, if we can prove that for any fixed cluster, the weight of the cluster decays exponentially with its size in some sense that we still have to clarify, and if we can prove that the number of clusters of fixed size grows at most exponentially with the size of the cluster, we can find an upper bound on the probability that all sites in some $\Lambda \subset \ent$ are positive at some time $n \in \nat$ with a simple geometric series. But before giving all the details of the Peierls argument, let us review some of the properties of $\var{\Lambda}$ and $\tnorm{\cdot}{\theta}$.

\subsection{Generalized Lasota-Yorke inequalities}

An important result for Interval Maps and Coupled Map Lattices is the Lasota-Yorke inequality \cite{LaYo73} which controls the growth of $\dnorm{\cdot}$ under the iterations of $T$. In this section, we will see that we can also control the growth of $\var{\Lambda}$ through a simple generalization of the usual Lasota-Yorke inequality.
\begin{proposition}[Generalized Lasota-Yorke inequalities]\label{prop2}
For every finite $\Lambda$ in $\ent$ and every $\mu \in L^1(X)$ such that $\var{\Lambda} \mu < \infty$, we have:
\begin{equation*}
\var{\Lambda} \left( T \mu\right) \leq  \sum_{\Omega \subseteq \Lambda} {\lambda_0}^{\norm{\Omega}}\ {D_0}^{\norm{\Lambda\setminus\Omega}} \ \var{\Omega}\mu.
\end{equation*}
\end{proposition}
\begin{proof}
Let $\varphi$ be some arbitrary function in $\mathcal{C}_{\Lambda}^1(X)$. Then, if $x_p$ is restricted to one of the intervals $J_i$, we see that $\varphi \circ T$ is differentiable with respect to $x_p$ and since we assumed that $\inf{\norm{\tau'}} >0$, we get:
\begin{align}
\charf{J_i}(x_p)\,(\partial_p \varphi )\circ T &  = \charf{J_i}(x_p)\,\frac{\partial_p( \varphi \circ T)}{\tau'(x_p)}\nonumber\\
& =\charf{J_i}(x_p)\  \partial_p\left(\frac{\varphi \circ T}{\tau'(x_p)}\right) - \charf{J_i}(x_p)\  \left({\frac{1}{\tau'}}\right)'\!\!(x_p)\;\varphi \circ T\label{eq10.2}.
\end{align}

Now, for every $p \in \ent$ and $i \in \{1,\ldots,N\}$, we introduce the operators $\Delta_{i,p}$ and $R_{i,p}$:
\begin{align}
\Delta_{i,p} \psi(\vect{x}) & = \frac{\psi(\zeta_{i+1},\vect{x}_{\neq p}) - \psi(\zeta_{i},\vect{x}_{\neq p})}{\zeta_{i+1}-\zeta_{i}} \label{eq10.2.0}\\ 
\begin{split}
R_{i,p}\psi(\vect{x}) & = \int_{\zeta_i}^{x_p} \Big(\partial_p\psi(\xi_p,\vect{x}_{\neq p}) - \Delta_{i,p} \psi(\xi_p,\vect{x}_{\neq p})\Big)\, d \xi_p \\
& =  \frac{\zeta_{i+1} - x_p}{\zeta_{i+1} -\zeta_{i}} \left(\psi(\vect{x}) - \psi(\zeta_{i},\vect{x}_{\neq p})\right) - \frac{x_p - \zeta_{i}}{\zeta_{i+1} -\zeta_{i}} \left(\psi(\zeta_{i+1},\vect{x}_{\neq p}) - \psi(\vect{x})\right)\label{eq10.2.1}
\end{split}
\end{align}
One might remark that if $\psi$ is a piecewise continuously differentiable function  with its discontinuities located at the boundaries of the intervals $J_i$, the function $\sum_i \charf{J_i}(x_p) R_{i,p}\psi(\vect{x})$ vanishes at the boundaries of the $J_i$ and is therefore not only piecewise continuously differentiable but also continuous with respect to $x_p$. Moreover, the definition of $R_{i,p}$ implies that:
\begin{equation}\label{eq10.3}
\charf{J_i}(x_p) \partial_p \psi = \charf{J_i}(x_p)  \partial_p(R_{i,p}\psi)  + \charf{J_i}(x_p) \Delta_{i,p}\psi.
\end{equation}

Therefore, using (\ref{eq10.3}) in (\ref{eq10.2}), we find:
\begin{align}
& \charf{J_i}(x_p)\,(\partial_p \varphi )\circ T\nonumber\\
 &  =   \charf{J_i}(x_p)\  \left(\partial_p R_{i,p}\left(\frac{\varphi \circ T}{\tau'(x_p)}\right) +   \Delta_{i,p}\left(\frac{\varphi \circ T}{\tau'(x_p)}\right) -  \left({\tfrac{1}{\tau'}}\right)'\!\!(x_p)\;\varphi \circ T\right) \nonumber\\
& =\charf{J_i}(x_p)\ \left(\partial_p \mathcal{K}_{i,p} + \mathcal{D}_{i,p}\right)(\varphi \circ T),
\end{align}
where the operators $\mathcal{K}_{i,p}$ and $\mathcal{D}_{i,p}$ are defined by:
\begin{equation}\label{eq15}\begin{split}
\mathcal{K}_{i,p}& : \psi \mapsto R_{i,p}\left(\frac{\psi}{\tau'(x_p)}\right)\\
\mathcal{D}_{i,p}& : \psi \mapsto \Delta_{i,p}\left(\frac{\psi}{\tau'(x_p)}\right) - \left({\frac{1}{\tau'}}\right)'\!\!(x_p)\;\psi 
\end{split}\end{equation}

For the proof of the usual Lasota-Yorke inequality, we just have to perform this construction for some fixed $p$ in $\ent$. But since we have multiple derivatives, we will iterate this for every $p$ in $\Lambda$. For any $\vect{i}_{\Lambda} = \{i_p\}_{p \in \Lambda}$, we define the set ${J(\vect{i}_{\Lambda}) = \{\,\vect{x}\ |\ \forall\, p \in \Lambda : x_p \in J_{i_p}\, \}}$. Since the operators $\partial_p$, $\mathcal{K}_{i,q}$ and $\mathcal{D}_{i,s}$ commute as long as $q$, $p$ and $s$ are different, we have:
\begin{align}
& \mu\left(\,\big(\,\partial_{\Lambda}\varphi\big)\circ T \,\right) =  \mu \left(\,\sum_{\vect{i}_{\Lambda}}\;\charf{J(\vect{i}_{\Lambda})}\,\big(\,\partial_{\Lambda}\;\varphi\big)\circ T \,\right)\nonumber\\
& \quad = \mu \left(\,\sum_{\vect{i}_{\Lambda}}\;\charf{J(\vect{i}_{\Lambda})}\,\prod_{p \in \Lambda} \left(\partial_p \mathcal{K}_{i_p,p} + \mathcal{D}_{i_p,p}\right)\;(\varphi \circ T)\,\right)\nonumber\\
& \quad = \sum_{\Omega \subseteq \Lambda} \mu \left(\,\sum_{\vect{i}_{\Lambda}}\;\charf{J(\vect{i}_{\Lambda})}\;\partial_{\Omega}\,\biggl[ \prod_{p \in \Omega} \mathcal{K}_{i_p,p}\biggr]\biggl[\prod_{p \in \Lambda \setminus \Omega} \mathcal{D}_{i_p,p}\biggr]\,(\varphi \circ T)\,\right)\label{eq16}.
\end{align}

If for every $\Omega \subseteq \Lambda$, we define the function:
\begin{equation*}
\psi_{\Omega} = \sum_{\vect{i}_{\Lambda}}\;\charf{J(\vect{i}_{\Lambda})}\;\biggl[ \prod_{p \in \Omega} \mathcal{K}_{i_p,p}\biggr]\biggl[\prod_{p \in \Lambda \setminus \Omega} \mathcal{D}_{i_p,p}\biggr]\,(\varphi \circ T),
\end{equation*}
we can see that, by definition of the operators $\mathcal{K}_{i_p,p}$, $\psi_{\Omega}$ vanishes when $x_p = \zeta_i$, for each $p \in \Omega$. Therefore, as long as $p$ is in $\Omega$, we have:
\begin{align}
\psi_{\Omega}(\vect{x}) = \sum_i \int_{-1}^{x_p} \charf{J_i}(\xi_p) \partial_p \psi_{\Omega}(\xi_p, \vect{x}_{\neq p}) \, \dif \xi_p.
\end{align}
Iterating this for every $p \in \Omega$ and taking the derivative with respect to all these variables yields:
\begin{equation}
\partial_{\Omega} \psi_{\Omega} = \sum_{\vect{i}_{\Lambda}}\;\charf{J(\vect{i}_{\Lambda})}\;\partial_{\Omega}\,\biggl[ \prod_{p \in \Omega} \mathcal{K}_{i_p,p}\biggr]\biggl[\prod_{p \in \Lambda \setminus \Omega} \mathcal{D}_{i_p,p}\biggr]\,(\varphi \circ T).
\end{equation}
Therefore, (\ref{eq16}) becomes:
\begin{equation*}
\mu\Bigl(\,\big(\,\partial_{\Lambda}\varphi\big)\circ T \,\Bigr)  = \sum_{\Omega \subseteq \Lambda} \mu \left(\,\partial_{\Omega}\,\psi_{\Omega}\,\right) 
\end{equation*}
but since $\psi_{\Omega}$ is piecewise continuous with respect to $\vect{x}_{\neq \Omega}$, continuous and piecewise continuously differentiable  with respect to $x_{\Omega}$, we can apply Proposition \ref{propA4} from the Appendix, and we get:
\begin{align}
& \mu\Bigl(\,\big(\,\partial_{\Lambda}\varphi\big)\circ T \,\Bigr) \leq \sum_{\Omega \subseteq \Lambda}\; \sup_{\vect{i}_{\Lambda}} \norm{\; \biggl[ \prod_{p \in \Omega} \mathcal{K}_{i_p,p}\biggr]\biggl[\prod_{p \in \Lambda \setminus \Omega} \mathcal{D}_{i_p,p}\biggr]\,(\varphi \circ T)\;}_{\infty} \;\var{\Omega}(\,\mu\,)\label{eq16.1}.
\end{align}
We can now check that for any continuous function $\psi$, by the definition of $\mathcal{K}_{i_p,p}$ and $\mathcal{D}_{i_p,p}$ from (\ref{eq15}), $R_{i_p,p}$ and $\Delta_{i_p,p}$ from (\ref{eq10.2.1}), and $\lambda_0$ and $D_0$  from the assumptions on $\tau$, we have:
\begin{equation}\label{eq17}
\norm{\mathcal{K}_{i_p,p} \psi}_{\infty} \leq \frac{1}{\kappa} \norm{R_{i_p,p} \psi}_{\infty} \leq  \frac{2}{\kappa} \norm{\psi}_{\infty} \leq \lambda_0 \norm{\psi}_{\infty}
\end{equation}
\begin{equation}\label{eq17.0}\begin{split}
\norm{\mathcal{D}_{i_p,p} \psi}_{\infty} & \leq \frac{1}{\kappa} \norm{\Delta_{i_p,p} \psi}_{\infty} + \norm{\left(\frac{1}{\tau'}\right)'}_\infty \norm{\psi}_{\infty}\\
& \leq \left(\frac{2}{\kappa\ \min_i \norm{J_{i}} }+\norm{\frac{\tau''}{\left(\tau'\right)^2}}_{\infty}\right) \norm{\psi}_{\infty} \leq D_0 \norm{\psi}_{\infty}
\end{split}
\end{equation}
Consequently, from (\ref{eq16.1}), we get the expected result:
\begin{equation*}
\mu\Bigl(\,\big(\,\partial_{\Lambda}\varphi\big)\circ T \,\Bigr)  \leq \norm{\varphi}_{\infty}\;\sum_{\Omega\subseteq \Lambda} {\lambda_0}^{\norm{\Omega}}\ {D_0}^{\norm{\Lambda \setminus \Omega}}\ \var{\Omega}\mu.
\end{equation*}
\end{proof}
A first consequence of Proposition \ref{prop2} is the Lasota-Yorke inequality. Indeed, if we take $\Lambda$ to be a singleton, and recall that $\sup_{p \in \ent} \var{\{p\}}\mu = \dnorm{\mu}$, we have:
\begin{equation}\label{eq17.1}
\dnorm{T \mu} = \sup_{p \in \ent}\ \var{\{p\}}T \mu \leq \lambda_0 \dnorm{\mu} + D_0 \norm{\mu}.
\end{equation}
This implies that the operator $T^t$ is uniformly bounded in $\mathcal{B}(X)$, because:
\begin{align}\label{eq17.2}
\dnorm{T^t \mu} & \leq \lambda_0^t \dnorm{\mu} + D_0 \sum_{k = 0}^{t-1}\  {\lambda_0}^k\  \norm{T^k \mu} \leq \lambda_0^t \dnorm{\mu} + D_0 \sum_{k = 0}^{t-1}\  {\lambda_0}^k\  \norm{\mu}\nonumber\\
& \leq \left({\lambda_0}^t + \tfrac{D_0}{1 - \lambda_0}\right)\, \dnorm{\mu}.
\end{align}
Therefore, if we take as initial measure $\sleb{-}$, the Lebesgue measure concentrated on $[-1,0]$, the sequence $\frac{1}{n} \sum_{t=0}^{n-1} T^t\sleb{-}$ is uniformly bounded in $\mathcal{B}(X)$ because ${\dnorm{\sleb{-}} = 2}$, and we can choose a subsequence which converges weakly to an invariant measure in $\mathcal{B}(X)$. Let $\muinv{-}$ be such an invariant measure.

Another important consequence of Proposition \ref{prop2} is the fact that the transfer operator $T$ is bounded in the $\theta$-norm, for $\theta$ large enough.
\begin{corollary}\label{prop4}
For any $\theta \geq\frac{D_0}{1 - \lambda_0}$ and any $\mu \in L^1(X)$ with bounded $\theta$-norm, we have: 
\begin{equation*}
\tnorm{T\mu}{\theta}  \leq \tnorm{\mu}{\theta}.
\end{equation*}
\end{corollary}
\begin{proof}
We use Proposition \ref{prop2} and the fact that $\theta^{- \norm{\Omega}} \var{\Omega}\mu \leq \tnorm{\mu}{\theta}$:
\begin{align*}
\theta^{- \norm{\Lambda}}\,\var{\Lambda}\left( T \mu\right) & \leq \theta^{ - \norm{\Lambda}} \;  \sum_{\Omega \subseteq \Lambda}\; \left(\lambda_0 \theta \right)^{\norm{\Omega}} {D_0}^{\norm{\Lambda} - \norm{\Omega}}\,\tnorm{\mu}{\theta}\\
& \leq \left( \lambda_0 + \tfrac{D_0}{\theta}\right)^{\norm{\Lambda}} \; \tnorm{\mu}{\theta}.
\end{align*}
The lower bound on $\theta$ then implies $\left( \lambda_0 + \tfrac{D_0}{\theta}\right) \leq 1$ and so $\tnorm{T\mu}{\theta}  \leq \tnorm{\mu}{\theta}$.\end{proof}

\subsection{The Peierls argument}

The bottom line of the Peierls argument is to show that the number of clusters of a fixed size grows at most exponentially with the size, and that the probability of having a large cluster decays exponentially with the size of the cluster. The first estimate, sometimes called the entropic estimate, is quite standard. However, the second estimate, also called the energetic estimate, will become problematic in the case of CML. Indeed, for any finite $\Lambda \subset \ent$, if ${E_{\epsilon}}^{\Lambda} \subseteq X$ is the set of configurations $\vect{x}$ such that $x_p \in E_{\epsilon}$ for any $p \in \Lambda$, we know that the Lebesgue measure of ${E_{\epsilon}}^{\Lambda}$ is smaller than $\norm{E_{\epsilon}}^{\raisebox{-1pt}{$\scriptstyle\Lambda$}}$, but we do not expect this to be true for an arbitrary signed measure, even if this measure is of bounded variation. For a measure of bounded variation, the best estimate one can find is $\lvert\charf{E_{\epsilon}^{\Lambda}} \mu \rvert \leq \norm{E_{\epsilon}}\,\dnorm{ \mu}$.

Therefore, we need to introduce extra regularity conditions on the initial measures. For instance, one could follow Bardet and Keller \cite{BaKe06} and consider only totally decoupled initial measures. But in order to prove the exponential convergence to equilibrium, we will need to apply the Peierls argument to an invariant measure which is not totally decoupled as long as $\epsilon \neq 0$. We will solve this problem in a new approach that relies on $\var{\Lambda}$ and the $\theta$-norm.

First, let us see how $\var{\Lambda}$ allows us to control $\lvert\charf{E_{\epsilon}^{\Lambda}} \mu \rvert$. If we define the operator $\mathcal{E}_{p}$ by:
\begin{equation}\label{eq7.2}
\mathcal{E}_{p} \psi(\vect{x}) = \int_{0}^{x_p} \charf{E_{\epsilon}}(\xi_p)\,\psi(\xi_p, \vect{x}_{\neq p})\, \mathrm d \xi_p.
\end{equation}
the symmetry assumption on $\tau$ implies that:
\begin{equation*} 
\norm{E_{\epsilon}\cap[0,1]} = \norm{E_{\epsilon}\cap[-1,0]} = \frac{\norm{E_{\epsilon}}}{2} 
\end{equation*}
and so:
\begin{equation}\label{eq7.3}
\norm{\mathcal{E}_{p} \psi}_{\infty} \leq \norm{\psi}_{\infty}\,\sup_{x_p \in [-1,1]} \norm{\int_{0}^{x_p} \charf{E_{\epsilon}}(\xi_p)\,\mathrm d \xi_p} \leq \tfrac{\norm{E_{\epsilon}}}{2}\,\norm{\psi}_{\infty}.
\end{equation}
We can check that $\charf{E_{\epsilon}}(x_p) \psi(\vect{x}) = \partial_p \mathcal{E}_{p} \psi(\vect{x})$. Hence, if $\Omega$ and $\Lambda$ are two disjoint subsets of $\ent$, we have:
\begin{align*}
\charf{{E_{\epsilon}}^{\Lambda}}\mu(\partial_{\Omega} \varphi) & = \mu(\,\charf{{E_{\epsilon}}^{\Lambda}} \,\partial_{\Omega}\varphi\,) =  \mu\left(\,\partial_{\Lambda \cup \Omega} \left[\prod_{p \in \Lambda}\mathcal{E}_{p}\right]\varphi\,\right) \leq \left(\tfrac{\norm{E_{\epsilon}}}{2}\right)^{\norm{\Lambda}} \, \var{\Omega\cup \Lambda}\mu \,\norm{\varphi}_{\infty}.
\end{align*}
This finally implies an estimate on $\charf{{E_{\epsilon}}^{\Lambda}}\mu$ with the appropriate exponential decay:
\begin{equation}\label{eq8}
\var{\Omega}\charf{{E_{\epsilon}}^{\Lambda}}\mu \leq \left(\tfrac{\norm{E_{\epsilon}}}{2}\right)^{\norm{\Lambda}} \, \var{\Omega\cup \Lambda}\mu . 
\end{equation}

However, if the assumption  $\Lambda \cap \Omega = \varnothing$ is not fulfilled, we can not use such a simple method without having to consider second derivatives with respect to some variables, which we do not expect to behave nicely. But the dynamic can help us, and with the generalized Lasota-Yorke inequalities, we have:
\begin{lemma}\label{prop6}For any measure $\mu$, any cluster $\Gamma$, any finite $\Omega \subseteq \ent$ and any $\Lambda \subseteq \Omega$, we have:
\begin{align*}
\var{\Omega} \left( T \charf{E(\Gamma,t)} \charf{{E_{\epsilon}}^{\Lambda}}\mu\right) \leq \sum_{V_1 \subseteq \Lambda } \; \sum_{V_0 \subseteq \Omega\setminus \Lambda } \; {\lambda_1}^{\norm{V_1}} {D_1}^{\norm{\Lambda \setminus V_1}}\,{\lambda_0}^{\norm{V_0}} \, {D_0}^{\norm{\Omega \setminus (\Lambda \cup V_0)}}\,\var{V_1 \cup V_0}\mu ,
\end{align*}
where $\lambda_1$ and $D_1$ were defined in (\ref{eq3.2}) and $E(\Gamma,t)$ in (\ref{eq5}).
\end{lemma}
\begin{proof}
We start by applying the development of (\ref{eq16}) to the measure $\charf{E(\Gamma,t)}\charf{{E_{\epsilon}}^{\Lambda}}\mu$:
\begin{align}
& T \charf{E(\Gamma,t)}\charf{{E_{\epsilon}}^{\Lambda}}\mu \left(\,   \partial_{\Omega} \varphi \, \right)\nonumber\\
& \qquad = \sum_{\vect{i}_{\Omega}}
\mu\left(\, \charf{E(\Gamma,t)}\ \charf{J(\vect{i}_{\Omega})}\ \charf{{E_{\epsilon}}^{\Lambda}}  \  \biggl[\prod_{p \in \Omega } ( \partial_p \mathcal{K}_{i_p,p} + \mathcal{D}_{i_p,p})\biggr](\varphi \circ T) \,\right)\label{eq20.2}.
\end{align}
We first consider the characteristic functions of $E(\Gamma,t)$ and $J(\vect{i}_{\Omega})$. Since the partition $J_i$ is finer than the intervals $(0,1]$ or $[-1,0]$, we know that if the configuration $\vect{x}$ is fixed outside $\Omega$, $\charf{E(\Gamma,t)}\,\charf{J(\vect{i}_{\Omega})} $ is either identically $0$ or identically $1$ as a function of $\vect{x}_{\Omega}$ restricted to $J(\vect{i}_{\Omega})$. Therefore, $\sum_{\vect{i}_{\Omega}}\charf{E(\Gamma,t)}\,\charf{J(\vect{i}_{\Omega})}$ can always be rewritten as $\sum_{\vect{i}_{\Omega}}
c_{\vect{i}_{\Omega}} \charf{J(\vect{i}_{\Omega})}$ where the $c_{\vect{i}_{\Omega}}$ are some discontinuous functions depending only on the variables outside $\Omega$ and taking only values $0$ and $1$. So (\ref{eq20.2}) can be rewritten as:
\begin{align}
& T \charf{E(\Gamma,t)}\charf{{E_{\epsilon}}^{\Lambda}}\mu \left(\,   \partial_{\Omega} \varphi \, \right)\nonumber\\
& \qquad =  \sum_{\vect{i}_{\Omega}}
\mu\left(\, c_{\vect{i}_{\Omega}}\ \charf{J(\vect{i}_{\Omega})}\ \charf{{E_{\epsilon}}^{\Lambda}}  \   \biggl[\prod_{p \in \Omega } ( \partial_p \mathcal{K}_{i_p,p} + \mathcal{D}_{i_p,p})\biggr](\varphi \circ T)\,\right)\label{eq18}.
\end{align}

We now focus on the characteristic function of ${E_{\epsilon}}^{\Lambda}$. Since $\Lambda \subseteq \Omega$, we have:
\begin{align}
\sum_{\vect{i}_{\Omega}}&  c_{\vect{i}_{\Omega}}
\charf{J(\vect{i}_{\Omega})}\,\charf{{E_{\epsilon}}^{\Lambda}}  \, \biggl[\prod_{p \in \Omega } ( \partial_p \mathcal{K}_{i_p,p} + \mathcal{D}_{i_p,p})\biggr]\,(\varphi \circ T) \nonumber\\
\begin{split} & = \sum_{\vect{i}_{\Omega}} c_{\vect{i}_{\Omega}}
\charf{J(\vect{i}_{\Omega})}\,  \biggl[\prod_{p \in \Lambda } (\charf{E_{\epsilon}}(x_p) \partial_p \mathcal{K}_{i_p,p} +  \charf{E_{\epsilon}}(x_p) \mathcal{D}_{i_p,p})\biggr] \\
& \phantom{= \sum_{\vect{i}_{\Omega}}\ } \biggl[\prod_{p \in \Omega \setminus \Lambda } ( \partial_p \mathcal{K}_{i_p,p} + \mathcal{D}_{i_p,p})\biggr]\,(\varphi \circ T) \label{eq20}.\end{split}
\end{align}
But $\charf{E_{\epsilon}}(x_p) \psi$ is equal to $\partial_p \mathcal{E}_{p} \psi$, with the operator $\mathcal{E}_{p}$ introduced in (\ref{eq7.2}), and therefore we have:
\begin{align*}
\charf{E_{\epsilon}}(x_p) \mathcal{D}_{i_p,p} \psi = \partial_p \mathcal{E}_{p} \mathcal{D}_{i_p,p} \psi. 
\end{align*}
And, since $\tau$ restricted to $J_i$ is monotone, we know that $E_{\epsilon} \cap J_i$ is always an interval, let us say $[a_i,b_i]$. Therefore, for any interval $J_{i}$ and any coordinate $p \in \ent$, we can define the operator:
\begin{equation}\label{eq19}
S_{i,p} \psi = \psi(a_i,\vect{x}_{\neq p}) + \int_{a_i}^{x_p} \charf{E_{\epsilon}} \partial_p \psi ,
\end{equation}
and we immediately see that, as long as $x_p$ belongs to $J_{i}$, $\partial_p S_{i,p} \psi =  \charf{E_{\epsilon}} \partial_p \psi$ and that ${\norm{ \charf{J_i}(x_p) S_{i,p} \psi}_{\infty} \leq \norm{\psi}_{\infty}}$. Hence, if $x_p \in J_{i_p}$:
\begin{align*}
\charf{E_{\epsilon}}(x_p) \partial_p \mathcal{K}_{i_p,p} \psi = \partial_p S_{i,p}  \mathcal{K}_{i_p,p} \psi.
\end{align*}

Eventually, equation (\ref{eq20}) can be rewritten as:
\begin{align*}
\sum_{\vect{i}_{\Omega}}&  c_{\vect{i}_{\Omega}}
\charf{J(\vect{i}_{\Omega})}\,\charf{{E_{\epsilon}}^{\Lambda}}  \, \biggl[\prod_{p \in \Omega } ( \partial_p \mathcal{K}_{i_p,p} + \mathcal{D}_{i_p,p})\biggr]\,(\varphi \circ T) \nonumber\\
\begin{split} & = \sum_{\vect{i}_{\Omega}} c_{\vect{i}_{\Omega}}
\charf{J(\vect{i}_{\Omega})}\,  \biggl[\prod_{p \in \Lambda }   (\partial_p S_{i_p,p} \mathcal{K}_{i_p,p} + \partial_p \mathcal{E}_{p} \mathcal{D}_{i_p,p})\biggr] \biggl[\prod_{p \in \Omega \setminus \Lambda } ( \partial_p \mathcal{K}_{i_p,p} + \mathcal{D}_{i_p,p})\biggr]\,(\varphi \circ T).\end{split}
\end{align*}
Here, we would like to apply directly Proposition \ref{propA4}, but this is impossible because the operator $S_{i_p,p}$ destroys the regularization introduced by $\mathcal{K}_{i_p,p}$. Indeed, if $\psi$ is continuously differentiable on the intervals $J_i$, the function $\sum_{i_p} \charf{J_i}(x_p) \mathcal{K}_{i_p,p}\psi(\vect{x})$ is not only continuously differentiable on the intervals $J_i$ but also continuous with respect to $x_p$, because $\mathcal{K}_{i_p,p}\psi(\vect{x})$ vanishes at the boundaries of the intervals $J_i$. However, this is no longer true for $\sum_{i_p} \charf{J_i}(x_p) S_{i_p,p} \mathcal{K}_{i_p,p}\psi(\vect{x})$, and we need to apply once again the operator $R_{i_p,p}$ from (\ref{eq10.2.1}) in order to regularize the discontinuities of the function. From  (\ref{eq10.3}), we have
 \begin{equation*}
\charf{J_{i_p}}\partial_p S_{i_p,p} \mathcal{K}_{i_p,p} \psi = \charf{J_{i_p}}\partial_p R_{i_p,p} S_{i_p,p} \mathcal{K}_{i_p,p} \psi + \charf{J_{i_p}}\Delta_{i_p,p} S_{i_p,p} \mathcal{K}_{i_p,p} \psi,  
 \end{equation*}
and this implies that
\begin{align}
\sum_{\vect{i}_{\Omega}}&  c_{\vect{i}_{\Omega}}
\charf{J(\vect{i}_{\Omega})}\,\charf{{E_{\epsilon}}^{\Lambda}}  \, \biggl[\prod_{p \in \Omega } ( \partial_p \mathcal{K}_{i_p,p} + \mathcal{D}_{i_p,p})\biggr]\,(\varphi \circ T) \nonumber\\
\begin{split} & = \sum_{\vect{i}_{\Omega}} c_{\vect{i}_{\Omega}}
\charf{J(\vect{i}_{\Omega})}\,  \biggl[\prod_{p \in \Lambda } ( \partial_p ( R_{i_p,p} S_{i_p,p} \mathcal{K}_{i_p,p} +  \mathcal{E}_{p} \mathcal{D}_{i_p,p}) +  \Delta_{i_p,p} S_{i_p,p} \mathcal{K}_{i_p,p})\biggr] \nonumber\\
& \phantom{= \sum_{\vect{i}_{\Omega}}\ } \biggl[\prod_{p \in \Omega \setminus \Lambda } \left( \partial_p \mathcal{K}_{i_p,p} + \mathcal{D}_{i_p,p}\right)\biggr]\,(\varphi \circ T)\end{split}\\
\begin{split}& = \sum_{V_1 \subseteq \Lambda} \;\sum_{V_0 \subseteq \Omega \setminus \Lambda}\; \sum_{\vect{i}_{\Omega}} c_{\vect{i}_{\Omega}}
\charf{J(\vect{i}_{\Omega})}\;  \biggl[\prod_{p \in V_1 } \partial_p (R_{i_p,p} S_{i_p,p} \mathcal{K}_{i_p,p} + \mathcal{E}_{p} \mathcal{D}_{i_p,p})\biggr] \\
& \phantom{= \sum_{\vect{i}_{\Omega}}\ }  \biggl[\prod_{p \in \Lambda \setminus V_1 } \Delta_{i_p,p} S_{i_p,p} \mathcal{K}_{i_p,p}\biggr] \biggl[\prod_{p \in V_0} \partial_p \mathcal{K}_{i_p,p}\biggr] \biggl[\prod_{ p \in \Omega \setminus (\Lambda \cup V_0) } \mathcal{D}_{i_p,p}\biggr] \, (\varphi \circ T) \nonumber\end{split}\\
& = \sum_{V_1 \subseteq \Lambda} \;\sum_{V_0 \subseteq \Omega \setminus \Lambda}\; \partial_{\raisebox{-3pt}{${\scriptstyle V_0 \cup V_1}$}}{\tilde{\varphi}}_{\raisebox{-3pt}{${\scriptstyle V_0 \cup V_1}$}},\label{eq20.1}
\end{align}
where, for $V_1$ and $V_0$ fixed, ${\tilde{\varphi}}_{V_0,V_1}$ is defined as:
\begin{equation}
\begin{split}{\tilde{\varphi}}_{V_0,V_1}&  = \sum_{\vect{i}_{\Omega}} c_{\vect{i}_{\Omega}}
\charf{J(\vect{i}_{\Omega})}\;  \biggl[\prod_{p \in V_1 }  (R_{i_p,p} S_{i_p,p} \mathcal{K}_{i_p,p} + \mathcal{E}_{p} \mathcal{D}_{i_p,p})\biggr] \\
& \phantom{= \sum_{\vect{i}_{\Omega}}\ }  \biggl[\prod_{p \in \Lambda \setminus V_1 } \Delta_{i_p,p} S_{i_p,p} \mathcal{K}_{i_p,p}\biggr] \biggl[\prod_{p \in V_0} \mathcal{K}_{i_p,p}\biggr] \biggl[\prod_{ p \in \Omega \setminus (\Lambda \cup V_0) } \mathcal{D}_{i_p,p}\biggr] \, (\varphi \circ T).  \label{eq20.3}\end{split}
\end{equation}

We can now conclude: if we insert (\ref{eq20.1}) into (\ref{eq18}), we find:
\begin{align*}
T \charf{E(\Gamma,t)}\charf{{E_{\epsilon}}^{\Lambda}}\mu \left(\,   \partial_{\Omega} \varphi \, \right) = \sum_{V_1 \subseteq \Lambda} \;\sum_{V_0 \subseteq \Omega \setminus \Lambda}\;\mu\left(\,\partial_{V_0 \cup V_1} \tilde{\varphi}_{V_0,V_1} \,\right) .
\end{align*}
But since $\tilde{\varphi}_{V_0,V_1}$ is continuous and piecewise continuously differentiable with respect to $x_{V_1 \cup V_0 }$, and piecewise continuous with respect to the other variables, we can apply Proposition \ref{propA4} and we get:
\begin{align}
T \charf{E(\Gamma,t)}\charf{{E_{\epsilon}}^{\Lambda}}\mu \left(\,   \partial_{\Omega} \varphi \, \right)\leq \sum_{V_1 \subseteq \Lambda } \; \sum_{V_0 \subseteq \Omega\setminus \Lambda } \;\var{V_1 \cup V_0}\mu\  \norm{\tilde{\varphi}_{V_0,V_1}}_{\infty}\label{eq20.4} .
\end{align}
Using $\norm{R_{i_p,p}\psi}_{\infty} \leq 2 \norm{\psi}_{\infty}$, $\norm{S_{i_p,p}\psi}_{\infty} \leq \norm{\psi}_{\infty}$, the bounds on $\mathcal{K}_{i_p,p}$ from (\ref{eq17}), on $\mathcal{D}_{i_p,p}$ from (\ref{eq17.0}) and on $\mathcal{E}_{p}$ from (\ref{eq7.3}), altogether with the definition of $\lambda_1$ and $D_1$ from (\ref{eq3.2}) yields
\begin{align*}
\norm{(R_{i_p,p} S_{i_p,p} \mathcal{K}_{i_p,p} +  \mathcal{E}_{p} \mathcal{D}_{i_p,p})\psi}_{\infty} & \leq \left(2 \,\frac{2}{\kappa} + \frac{\norm{E_{\epsilon}}}{2}\, D_0\right) \norm{\psi}_{\infty} \leq \lambda_1 \norm{\psi}_{\infty}\\
\norm{\Delta_{i_p,p} \mathcal{K}_{i_p,p}\psi}_{\infty} & \leq \frac{2}{\min_i \norm{J_i}} \,\frac{2}{\kappa} \leq D_1 \norm{\psi}_{\infty} .
\end{align*}
Therefore, $\norm{\tilde{\varphi}_{V_0,V_1}}_{\infty}$ is bounded by:
\begin{equation*}
 \norm{\tilde{\varphi}_{V_0,V_1}}_{\infty} \leq \lambda_1^{\norm{V_1}} D_1^{\norm{\Lambda \setminus V_1}}\,\lambda_0^{\norm{V_0}} \, D_0^{\norm{\Omega \setminus (\Lambda \cup V_0)}} \ \norm{\varphi}_{\infty} .
\end{equation*}
And so, (\ref{eq20.4}) becomes:
\begin{align*}
& T \charf{E(\Gamma,t)}\charf{{E_{\epsilon}}^{\Lambda}}\mu \left(\,   \partial_{\Omega} \varphi \, \right)\nonumber\\
& \qquad \leq \norm{\varphi}_{\infty}\,\sum_{V_1 \subseteq \Lambda } \; \sum_{V_0 \subseteq \Omega\setminus \Lambda } \; \lambda_1^{\norm{V_1}} D_1^{\norm{\Lambda \setminus V_1}}\,\lambda_0^{\norm{V_0}} \, D_0^{\norm{\Omega \setminus (\Lambda \cup V_0)}}\,\var{V_1 \cup V_0}\mu.
\end{align*}
\end{proof}

We have now all the tools to complete the Peierls argument.
\begin{theorem}\label{prop7}Assume that $\tau$ belongs to $\mathcal{T}(D_0,c,\varsigma)$ and that $\kappa > 108$. Then, there is some $\epsilon_0 > 0$ such that, if $\epsilon < \epsilon_0$, for any measure $\mu$ in $\mathcal{B}(K,\alpha,\theta_0)$ with $K < \infty$, $\alpha < \frac{1}{27}$ and $\theta_0 = \frac{2 \alpha_0}{\norm{E_{\epsilon}}}$, we have:
\begin{equation*}
\tnorm{\,\charf{(0,1]^{\Lambda}}\,T^n \mu\,}{\theta_0} \leq \frac{K}{2 (1 - 9 \alpha')(1 -\alpha')}\,{\alpha'}^{\norm{\Lambda}},
\end{equation*}
with $\alpha' = 3\,\max \{ \alpha_0, \alpha\} < \frac{1}{9}$ and $\alpha_0$ defined in (\ref{eq3.3}).
\end{theorem}
\begin{proof}
We start with the contour expansion of (\ref{eq6}):
\begin{align}
\tnorm{\,\charf{(0,1]^{\Lambda}}\,T^n \mu\,}{\theta_0}  & \leq  \sum_{\Gamma \in \mathcal{G}(\Lambda)}\,\Big\vvvert\,\charf{(0,1]^{\Lambda}}\,\Bigl[\prod_{t = 0}^{n-1} \left(\, T \charf{E(\Gamma,t)}\,\right)\Bigr]  \mu \,\Big\vvvert_{\theta_0}\label{eq22.-1}.
\end{align}
However, for some fixed cluster $\Gamma$, we know by the definition of $E(\Gamma,t)$ in (\ref{eq5}) that if $\vect{x}$ belongs to $E(\Gamma,t)$ and $T \vect{x}$ belongs to $E(\Gamma,t+1)$, $\vect{x}$ has to be in $E_{\epsilon}^{\partial \Gamma_{t+1}}$ with  $\partial \Gamma_{t+1} = \{\, q\ |\ (q,t+1) \in \partial \Gamma\, \}$. So:
\begin{equation*}
\charf{E(\Gamma,t+1)}(T\vect{x})\ \charf{E(\Gamma,t)}(\vect{x}) = \charf{E(\Gamma,t+1)}(T\vect{x})\ \charf{E(\Gamma,t)}(\vect{x})\ \charf{{E_{\epsilon}}^{\partial \Gamma_{t+1}}}(\vect{x}).
\end{equation*}
Therefore, we can insert the characteristic functions of ${E_{\epsilon}}^{\partial \Gamma_{t+1}}$ at every time $t$ in each term of the sum in (\ref{eq22.-1}), and we get:
\begin{align}
\tnorm{\,\charf{(0,1]^{\Lambda}}\,T^n \mu\,}{\theta_0}  & \leq  \sum_{\Gamma \in \mathcal{G}(\Lambda)}\,\Big\vvvert\,\charf{(0,1]^{\Lambda}}\,\Bigl[\prod_{t = 0}^{n-1} \left(\, T \charf{E_{\epsilon}^{\partial \Gamma_{t+1}}}\charf{E(\Gamma,t)}\,\right)\Bigr]  \mu \,\Big\vvvert_{\theta_0}\label{eq22}.
\end{align}
But, for any measure $\nu$ and any finite subset $\Omega$, if we first apply Lemma \ref{prop6} and then use inequality (\ref{eq8}), we have:
{\allowdisplaybreaks
\begin{align}
\theta_0^{- \norm{\Omega}}\,& \var{\Omega}\left( T \charf{E_{\epsilon}^{\partial \Gamma_{t}}}\charf{E(\Gamma,t-1)} \nu\right) \nonumber\\
\begin{split}&\leq \theta_0^{- \norm{\Omega}}\,\sum_{V_1 \subseteq \Omega \cap \partial \Gamma_{t} }\; \sum_{V_0 \subseteq \Omega\setminus \partial \Gamma_{t} }\; \lambda_1^{\norm{V_1}}\, D_1^{\norm{(\Omega \cap \partial \Gamma_{t}) \setminus V_1}}\,\lambda_0^{\norm{V_0}}\,D_0^{\norm{\Omega \setminus (\partial \Gamma_{t} \cup V_0)}}\nonumber\\
&\phantom{\leq \theta_0^{- \norm{\Omega}}\,\sum_{V_1 \subseteq \Omega \cap \partial \Gamma_{t} }\; \sum_{V_0 \subseteq \Omega\setminus \partial \Gamma_{t}}\;}\var{V_1 \cup V_0} (\charf{E_{\epsilon}^{\partial \Gamma_{t} \setminus \Omega}}\nu)\end{split}\\
\begin{split} &\leq \theta_0^{- \norm{\Omega}}\,\sum_{V_1 \subseteq \Omega \cap \partial \Gamma_{t} }\; \sum_{V_0 \subseteq \Omega\setminus \partial \Gamma_{t} } \;  \lambda_1^{\norm{V_1}}\, D_1^{\norm{(\Omega \cap \partial \Gamma_{t}) \setminus V_1}}\,\lambda_0^{\norm{V_0}}\, D_0^{\norm{\Omega \setminus (\partial \Gamma_{t}\cup V_0)}}\nonumber\\
& \phantom{\leq \theta_0^{- \norm{\Omega}}\,\sum_{V_1 \subseteq \Omega \cap \partial \Gamma_{t}}\; \sum_{V_0 \subseteq \Omega\setminus \partial \Gamma_{t} }\;}\left(\frac{\norm{E_{\epsilon}}}{2}\right)^{\norm{\partial \Gamma_{t} \setminus \Omega}}\,\var{V_1 \cup V_0 \cup (\partial \Gamma_{t} \setminus \Omega)} (\nu) \end{split}\nonumber\\
\begin{split} & \leq \theta_0^{- \norm{\Omega}}\,\sum_{V_1 \subseteq \Omega \cap \partial \Gamma_{t} } \sum_{V_0 \subseteq \Omega\setminus \partial \Gamma_{t} }  \left(\lambda_1 \theta_0\right)^{\norm{V_1}} D_1^{\norm{(\Omega \cap \partial \Gamma_{t}) \setminus V_1}} \left(\lambda_0 \theta_0\right)^{\norm{V_0}} D_0^{\norm{\Omega \setminus (\partial \Gamma_{t} \cup V_0)}}\nonumber\\
& \phantom{\leq \theta_0^{- \norm{\Omega}}\,\sum_{V_1 \subseteq \Omega \cap \partial \Gamma_{t} }\; \sum_{V_0 \subseteq \Omega\setminus \partial \Gamma_{t} }\;}\left(\frac{\theta_0\norm{E_{\epsilon}} }{2}\right)^{\norm{\partial \Gamma_{t} \setminus \Omega}}\,\tnorm{\nu}{\theta_0}\end{split}\nonumber\\
& \leq \left(\lambda_1 + \frac{D_1 }{\theta_0}\right)^{\norm{\Omega \cap \partial \Gamma_{t}} } \left(\lambda_0 + \frac{D_0}{\theta_0} \right)^{\norm{\Omega\setminus \partial \Gamma_{t}} }\,\left(\frac{\theta_0 \norm{E_{\epsilon}}}{2}\right)^{\norm{\partial \Gamma_{t} \setminus \Omega}}\,\tnorm{\nu}{\theta_0}\label{eq23}.
\end{align}}
By definition of $\theta_0$, $\frac{\theta_0 \norm{E_{\epsilon}}}{2} \leq \alpha_0$, and we can check that:
\begin{align*}
\lambda_1 + \frac{D_1 }{\theta_0} \leq \alpha_0 &\iff \alpha_0 \geq \frac{1}{2} \left( \lambda_1 + \sqrt{\lambda_1^2 + 2 D_1 \norm{E_{\epsilon}}}\right)\\
\lambda_0 + \frac{D_0 }{\theta_0} \leq 1 & \iff \alpha_0 \geq \frac{1}{2}\left(\frac{D_0 \norm{E_{\epsilon}}}{1 - \lambda_0}\right)
\end{align*}
So, the definition of $\alpha_0$ in (\ref{eq3.3}) implies that, if we take the supremum over all finite $\Omega$ in (\ref{eq23}), we get:
\begin{equation}\label{eq24}
\tnorm{ T \charf{E_{\epsilon}^{\Lambda_{t}}}\charf{E(\Gamma,t-1)} \nu}{\theta_0} \leq {\alpha_0}^{\norm{\partial \Gamma_{t}}} \tnorm{\nu}{\theta_0}.
\end{equation}

We now go back to equation (\ref{eq22}). We apply Corollary \ref{propA5}, inequality (\ref{eq24}), use the assumption that $\mu$ belongs to $\mathcal{B}(K,\alpha,\theta_0)$, define $\alpha' = 3 \max\{\alpha_0,  \alpha\}$, recall the definition of $\partial \Gamma$, and we find:

\begin{align}
\tnorm{\,\charf{(0,1]^{\Lambda}}\,T^n \mu\,}{\theta_0} & \leq \sum_{\Gamma \in \mathcal{G}(\Lambda)}\, \Big\vvvert\,\,\charf{(0,1]^{\Lambda}}\,\Bigl[\prod_{t = 0}^{n-1} \left(\, T \charf{E(\Gamma,t)}\,\right)\Bigr]  \mu \,\Big\vvvert_{\theta_0}\nonumber\\
&\leq \sum_{\Gamma \in \mathcal{G}(\Lambda)}\,\Big\vvvert\,\prod_{t = 0}^{n-1} \left(\, T \charf{E_{\epsilon}^{\partial \Gamma_{t+1}}}\charf{E(\Gamma,t)}\,\right)  \charf{(0,1]^{\Lambda_0}}  \mu \,\Big\vvvert_{\theta_0}\nonumber\\
&\leq \sum_{\Gamma \in \mathcal{G}(\Lambda)}\,{\alpha_0}^{\sum_{t=1}^{n} \norm{\partial \Gamma_{t}}}\, \tnorm{ \charf{(0,1]^{\partial \Gamma_0}}  \mu\,}{\theta_0} \label{eq25.-1}\\
&\leq \sum_{\Gamma \in \mathcal{G}(\Lambda)}\,K {\alpha_0}^{\sum_{t=1}^{n} \norm{\partial \Gamma_{t}}} {\alpha}^{\norm{\partial \Gamma_0}}\nonumber\\
&\leq  \sum_{\Gamma \in \mathcal{G}(\Lambda)}\,K {\left(\frac{\alpha'}{3}\right)}^{\norm{\partial \Gamma}}\label{eq26}.
\end{align}

We can now count the number of clusters. A cluster $\Gamma$ is univoquely determined by its outer path and there are at most $3^{n_d + n_v + n_h}$ outer paths with $n_d$, $n_v$ and $n_h$ edges in the diagonal, vertical and horizontal directions respectively. We have seen in (\ref{eq7}) that $n_d = n_v$ and that there is some $k \geq 0$ such that $n_h = n_d + \norm{\Lambda} - c + k$ and $\norm{\partial \Gamma} \geq n_h + c = n_d + \norm{\Lambda} + k$. Therefore, (\ref{eq26}) becomes:
\begin{equation}\label{eq26.1}
\tnorm{\charf{(0,1]^{\Lambda}}\,T^n \mu}{\theta_0}  \leq K\,\sum_{n_d = 0}^{\infty}\,\sum_{k = 0}^{\infty}\,\sum_{c = 1}^{\infty} \,3^{3 n_d + \norm{\Lambda} -c + k }\, {\left(\frac{\alpha'}{3}\right)}^{n_d + \norm{\Lambda} +k}.
\end{equation}
Now, we have seen in (\ref{eq22.-1}) that $\lim_{\epsilon \to 0} \alpha_0 = \frac{4}{\kappa}$. If we assume that $\kappa > 108$, there is some $\epsilon_0$ such that $\epsilon < \epsilon_0$ implies $\alpha_0 < \frac{1}{27}$. We assume then that $\kappa > 108$ and $\epsilon < \epsilon_0$. Since we also assumed that $\alpha < \frac{1}{27}$, we have $\alpha' < \frac{1}{9}$ and the geometric series  in (\ref{eq26.1}) converge. This yields:
\begin{equation*}
\tnorm{\charf{(0,1]^{\Lambda}}\,T^n \mu}{\theta_0}  \leq \frac{K}{2(1 - 9 \alpha')(1- \alpha')} \,{\alpha'}^{\norm{\Lambda}}.
\end{equation*}
\end{proof}

Now, Theorem \ref{result1} is a straightforward consequence of Theorem \ref{prop7}.

\begin{proof}[Theorem \ref{result1}] 

Let us start by proving that $\muinv{-}$ belongs to $\mathcal{B}(K_0,3 \alpha_0,\theta_0)$. Since $\var{\Lambda} \sleb{-} \leq 2^{\norm{\Lambda}}$, we know that the measure $\sleb{-}$ belongs to $\mathcal{B}(1,0,2)$. But, by definition of $\theta_0$, we have:
\begin{equation}
\theta_0 =  \frac{2 \alpha_0}{\norm{E_{\epsilon}}} \geq \frac{D_0}{1 - \lambda_0} \geq 2 \label{eq27}.
\end{equation}
Therefore, $\sleb{-}$ also belongs to $\mathcal{B}(1,0,\theta_0)$. We can now apply Theorem \ref{prop7}. If $\kappa > 108$, there is some $\epsilon_0$ such that, if $\epsilon < \epsilon_0$,  $T^t \sleb{-}$ belongs to $\mathcal{B}(K_0,3 \alpha_0,\theta_0)$ for any $t$ and $3 \alpha_0 < \frac{1}{9}$. Therefore, $\muinv{-}$ too belongs to $\mathcal{B}(K_0, 3 \alpha_0,\theta_0)$. Since $\muinv{+}$ does not belong to any $\mathcal{B}(K,\alpha,\theta)$ as long as $\alpha < 1$, this also proves that $\muinv{+} \neq \muinv{-}$.\end{proof}

\section{Exponential convergence to equilibrium}
\label{expconv}

In the previous section, we defined the measure $\muinv{-}$ as a converging subsequence of $\frac{1}{n} \sum_{t=0}^{n-1} T^t\sleb{-}$. However, it was actually unnecessary to take the limit in the sense of Cesaro and to restrict ourselves to some subsequence, because, as we will see in this section, $\sleb{-}$ and many other initial probability measures converge exponentially fast to $\muinv{-}$. 

Let us start by choosing some arbitrary positive integer $\gamma$ and considering the well-ordering $\prec$, defined by:
\begin{equation*}
0 \prec -1 \prec -2 \prec \ldots \prec -\gamma \prec +1 \prec -\gamma-1 \prec +2 \prec -\gamma -2 \ldots 
\end{equation*}
With this ordering, we see that all the sites influenced by $0$ after $\gamma$ iterations of the dynamic are the $\gamma + 1$ first sites. Let $\prec_q$ be the translation of this well-ordering at any site $q$ of $\ent$. Then, for any $q$ and $p$ in $\ent$, we define the operator $\proj{qp}$:
\begin{equation}\label{eq28}
\proj{qp}\mu (\varphi) \equiv \mu (\proj{qp}\varphi) = \mu \left(\int \varphi \prod_{s \prec_q p }  \hinv{-}(x_s)\ \mathrm{d}x_s - \int \varphi \prod_{s \preceq_q p } \hinv{-}(x_s)\ \mathrm{d}x_s \right). 
\end{equation}
where $\hinv{-}$ was defined as the invariant measure of the local map $\tau$ concentrated on $[-1,0]$. Note that as long as $\varphi$ does not depend on $x_p$, $\proj{qp}\varphi$ is identically zero. Furthermore, for any $\varphi$ depending only on the variables in $\Lambda$, an arbitrary finite subset of $\ent$, and for any signed measure of zero mass $\mu$, we have, for any $q \in \ent$:
\begin{align*}
\sum_{p \in \ent} \proj{qp} \mu(\varphi) & = \sum_{p \in \ent} \mu \left(\int \varphi \prod_{s \prec_q p }  \hinv{-}(x_s)\ \mathrm{d}x_s - \int \varphi \prod_{s \preceq_q p } \hinv{-}(x_s)\ \mathrm{d}x_s \right)\\
& = \mu(\varphi) - \mu(1)\ \int \varphi \prod_{s \in \Lambda } \hinv{-}(x_s)\ \mathrm{d}x_s = \mu(\varphi).
\end{align*}
Since this is true for any continuous function, this implies that any signed measure of zero mass $\mu$ can be decomposed as:
\begin{equation}\label{eq29}
\mu = \sum_{p \in \ent} \proj{qp} \mu .
\end{equation}

We can also see that the operator $\proj{qp}$ is bounded both in total variation norm, bounded variation norm and $\theta$-norm, as stated in the next lemma:
\begin{lemma}\label{prop8}
For any measure $\mu$ in $\mathcal{B}(X)$, any $q,p \in \ent$ and any $\theta \geq \frac{D_0}{1 - \lambda_0}$, we have :
\begin{equation*}
\begin{cases}&\norm{\proj{qp}\mu} \leq 2\ \norm{\mu}\\
&\dnorm{\proj{qp}\mu} \leq 2\  \left(\tfrac{D_0}{1- \lambda_0}\right) \ \dnorm{\mu}\\
&\tnorm{\proj{qp}\mu}{\theta} \leq 2 \tnorm{\mu}{\theta}\end{cases}
\end{equation*}
\end{lemma}
\begin{proof}For the two first inequalities, it is sufficient to prove that, for any $\Lambda \subset \ent$, we have:
\begin{align*}
\var{\Lambda}\proj{qp}\mu  \leq 2 \left(\tfrac{D_0}{1- \lambda_0}\right)^{\norm{\Lambda}} \var{\Lambda}(\mu).
\end{align*}
So, for any function $\varphi$ in $\mathcal{C}_{\Lambda}^1(X)$ with $\norm{\varphi}_{\infty} \leq 1$, we consider:
\begin{align}\label{eq30}
\proj{qp}\mu(\partial_{\Lambda} \varphi) = \mu \left(\int \partial_{\Lambda}\varphi \prod_{s \prec_q p }  \hinv{-}(x_s)\ \mathrm{d}x_s\right) - \mu \left(\int \partial_{\Lambda} \varphi \prod_{s \preceq_q p }  \hinv{-}(x_s)\ \mathrm{d}x_s\right) .
\end{align}
If we define $\Omega = \{\,k \in \Lambda\ |\ q \preceq_q k \prec_q p\,\}$ , we see that the derivatives with respect $\vect{x}_{\Lambda \setminus \Omega}$ commute with the first integral of (\ref{eq30}). And the same can be done for the second integral of (\ref{eq30}) with $\Omega' = \{\,k \in \Lambda\ |\ q \preceq_q k \preceq_q p\,\}$. And so, by definition of $\var{\Omega}$ or $\var{\Omega'}$, we get:
\begin{align}
\proj{qp}\mu(\partial_{\Lambda} \varphi) & = \mu \left(\partial_{\Lambda\setminus \Omega}\int \partial_{\Omega}\varphi \prod_{s \prec_q p } \hinv{-}(x_s)\ \mathrm{d}x_s\right)\nonumber\\
&\qquad  - \mu \left(\partial_{\Lambda\setminus \Omega'}\int \partial_{\Omega'} \varphi \prod_{s \preceq_q p }  \hinv{-}(x_s)\ \mathrm{d}x_s\right)\nonumber\\ 
& \leq \var{\Lambda \setminus \Omega}(\mu) \ \norm{\int \partial_{\Omega}\varphi \prod_{s \prec_q p } \hinv{-}(x_s)\ \mathrm{d}x_s}_{\infty} \nonumber\\
&\qquad +\var{\Lambda \setminus \Omega'}(\mu) \norm{\int \partial_{\Omega'} \varphi \prod_{s \preceq_q p }  \hinv{-}(x_s)\ \mathrm{d}x_s}_{\infty}\nonumber\\ 
&\leq \var{\Lambda \setminus\Omega}(\mu)\  \dnorm{\hinv{-}}_{BV}^{\norm{\Omega}}\,\norm{\varphi}_{\infty} + \var{\Lambda \setminus \Omega'}(\mu)\ \dnorm{\hinv{-}}_{BV}^{\norm{\Omega'}}\,\norm{\varphi}_{\infty}\label{eq31}.
\end{align}

We can find an upper bound on $\dnorm{\hinv{-}}_{BV}$ with the the Lasota-Yorke inequality of $\tau$ from \ref{eq0.1}. Indeed, if $h_{\mathrm{leb}}^{{\scriptscriptstyle(-)}}$ is the Lebesgue measure concentrated on $[-1,0]$, we have:
\begin{equation*}
\lVert\tau^n h_{\mathrm{leb}}^{{\scriptscriptstyle(-)}}\rVert_{BV} \leq \lambda_0^n \lVert h_{\mathrm{leb}}^{{\scriptscriptstyle(-)}} \rVert_{BV} + \frac{D_0}{1 - \lambda_0} \lvert h_{\mathrm{leb}}^{(-)}\rvert_{L^1(I)},
\end{equation*}
and it implies that $\dnorm{\hinv{-}}_{BV} \leq \frac{D_0}{1 - \lambda_0}$ because $\tau^n h_{\mathrm{leb}}^{{\scriptscriptstyle(-)}}$ converges to $\hinv{-}$. Hence, since $\var{\Lambda \setminus\Omega} \mu \leq \var{\Lambda} \mu$, $\Omega \subseteq \Lambda$ and since we already saw that $\frac{D_0}{1 - \lambda_0} \geq 1$, (\ref{eq31}) becomes:
\begin{align*}
\var{\Lambda}\proj{qp}\mu & \leq \var{\Lambda \setminus\Omega}(\mu)\  \left(\tfrac{D_0}{1- \lambda_0}\right)^{\norm{\Omega}} + \var{\Lambda \setminus \Omega'}(\mu)\ \left(\tfrac{D_0}{1- \lambda_0}\right)^{\norm{\Omega'}}\\
&  \leq \var{\Lambda}(\mu)\  \left(\tfrac{D_0}{1- \lambda_0}\right)^{\norm{\Omega}} + \var{\Lambda }(\mu)\ \left(\tfrac{D_0}{1- \lambda_0}\right)^{\norm{\Omega'}}\\
&  \leq 2 \left(\tfrac{D_0}{1- \lambda_0}\right)^{\norm{\Lambda}}\ \var{\Lambda }(\mu),
\end{align*}
and this proves the two first inequalities of the Lemma.

The bound in the $\theta$-norm is a consequence of (\ref{eq31}). Indeed, if we multiply each side of the inequality by $\theta^{- \norm{\Lambda}}$, and use $\theta \geq \frac{D_0}{1-\lambda_0} \geq \dnorm{\hinv{-}}$, we get:
\begin{equation*}
\tnorm{\proj{qp}\mu}{\theta} \leq \tnorm{\mu}{\theta} + \tnorm{\mu}{\theta}.
\end{equation*}
\end{proof}

Let us now consider some signed measure of zero mass $\mu$ and some continuous function $\varphi$ on $X$ depending only on the variables $x_\Lambda$ for some finite $\Lambda$ in $\ent$, and carry out the decomposition of (\ref{eq29}) after every $\gamma$ iteration of the dynamic. If we assume that $t = m\, \gamma$ for some $m \in \nat$, we get:
\begin{equation*}
T^t \mu (\varphi) =   \sum_{\vect{p} \in \ent^{m+1}} \proj{p_{m-1} p_{m}}T^{\gamma} \proj{p_{m-2} p_{m-1}}\ldots T^{\gamma}\proj{0 p_{0}}\mu(\varphi).
\end{equation*}
If $\mathcal{F}(\Lambda)$ is the set of $\vect{p}$ such that $p_m \in \Lambda$ and every $p_{k}$ belongs to ${\{p_{k-1}- \gamma,\ldots,p_{k-1}\}}$, we notice that every configuration of $\vect{p}$ not belonging to $\mathcal{F}(\Lambda)$ does not contribute to the sum. Indeed, if $p_m$ does not belong to $\Lambda$, $\proj{p_{m-1} p_{m}} \varphi$ is identically zero because $\varphi$ does not depend on $x_{p_m}$, so $\proj{p_{m-1} p_{m}}\nu(\varphi) = \nu(\proj{p_{m-1} p_{m}}\varphi) = 0$ for any measure $\nu$. And if $p$ does not belong to $\{q - \gamma,\ldots,q\}$, we see that, from the definition of $\prec$,  $\proj{q p}\varphi$ does not depend on any variable $x_{s}$ with $s \in \{q - \gamma,\ldots,q\}$, and so $\left(\proj{q p}\varphi\right)\circ T^{\gamma}$ does not depend on $x_q$. Therefore, $\proj{q p} T^{\gamma} \proj{sq}\nu = 0$ for any $s \in \ent$ and any measure $\nu$. Hence, if we define:
\begin{equation}
\tilde{T}^{qp} = T^{\gamma} \proj{q p} \label{eq33.-1} ,
\end{equation}
and apply Lemma \ref{prop8}, the expansion becomes:
\begin{align}
\norm{T^t \mu (\varphi)}& \leq \sum_{\vect{p} \in \mathcal{F}(\Lambda)} \norm{\proj{p_{m-1} p_{m}}T^{\gamma} \proj{p_{m-2} p_{m-1}}\ldots T^{\gamma}\proj{0 p_{0}}\mu(\varphi)}\nonumber\\
& \leq 2\,\sum_{\vect{p} \in \mathcal{F}(\Lambda)} \norm{\tilde{T}^{p_{m-2} p_{m-1}}\ldots \tilde{T}^{0 p_{0}}\mu}\,\norm{\varphi}_{\infty}\label{eq33}.
\end{align}
In the next subsection, we will prove with a decoupling argument that the dynamic restricted to a pure phase, namely the operator $\tilde{T}^{qp}\charf{[-1,0]}(x_p) \charf{[-1,0]}(x_{p+1})$, is a contraction in $\mathcal{B}(X)$.

\subsection{Decoupling in the pure phases}
\label{decpurphas}

The idea behind the decoupling in the pure phase is to reproduce the decoupling argument of Keller and Liverani \cite{KeLi05,KeLi06}, but instead of considering the coupling as a perturbation of the identity, we will consider the coupling as a perturbation of a strongly coupled dynamic for which we can prove the exponential convergence to equilibrium in the pure phases. The decoupled dynamic at site $p$ is given by $T_{0}^{(p)} = \Phi_{0}^{(p)} \circ \tau^{  \ent}$, where the coupling $\Phi_{0}^{(p)}$ is explicitly given by:
\begin{equation}\label{eq36}
\Phi_{0,q}^{(p)}(\vect{x}) = 
\begin{cases} \Phi_{\epsilon,q}(\vect{x}) \;&
\textrm{if}\; q \neq p\\
\Phi_{0,p}(\vect{x}_{\neq p+1}, -1)\;&\textrm{if}\; q = p\\
\end{cases}
\end{equation}

The next proposition shows that this slight modification of the coupling does not change too much the dynamic when applied to a measure $\charf{[-1,0]}(x_{p+1})\mu$.

\begin{proposition}\label{prop9} Let $\mu \in \mathcal{B}(X)$ and $p \in \ent$. Then:
\begin{equation*}
\norm{\,\Phi_{\epsilon}\,\charf{[-1,0]}(x_{p+1})\,\mu  - \Phi_{0}^{(p)}\,\charf{[-1,0]}(x_{p+1})\,\mu\,} \leq  2\,\epsilon\, \dnorm{\mu}.
\end{equation*}
\end{proposition}
\begin{proof}For the demonstration of this Proposition, we will basically follow the lines of the proof of Proposition 5 in \cite{KeLi05}. If  $F_t = t\, \Phi_{\epsilon} + (1-t)\, \Phi_{0}^{(p)}$, we can state that:
\begin{align}
\mu\Big(\charf{[-1,0]}(x_{p+1}) \, \varphi \circ\Phi_{\epsilon}\Big)&  - \mu\Big(\charf{[-1,0]}(x_{p+1}) \, \varphi \circ\Phi_{0}^{(p)}\Big)\nonumber\\
 & = \mu\left( \charf{[-1,0]}(x_{p+1}) \int_{0}^{1} \partial_t \left(\varphi \circ F_t\right) \dif{t}\right)\nonumber\\
& = \int_{0}^{1} \,\mu\left( \sum_{q} \charf{[-1,0]}(x_{p+1})\;\left(\partial_q \varphi \circ F_t\right) \; \partial_t F_{t,q}  \right) \, \dif{t} \label{eq37}.
\end{align}
But we can check that $\charf{[-1,0]}(x_{p+1})\,\partial_t F_{t,q} = \charf{[-1,0]}(x_{p+1})\,(\Phi_{\epsilon,q}(\vect{x}) - \Phi_{0,q}^{(p)}(\vect{x}))$ is equal to $0$ if $q \neq p$, and to $\charf{[-1,0]}(x_{p+1})\,\epsilon$ if $q = p$. So, the sum over $q$ reduces to the term $q=p$  and becomes:
\begin{align}
&\mu\Big(\charf{[-1,0]}(x_{p+1}) \, \varphi \circ\Phi_{\epsilon}\Big)  - \mu\Big(\charf{[-1,0]}(x_{p+1}) \, \varphi \circ\Phi_{0}^{(q)}\Big)\nonumber\\
& \rule{30pt}{0pt} = \epsilon \;\int_{0}^{1} \,\mu\left(\,\charf{[-1,0]}(x_{p+1})\,\partial_p \varphi \circ F_t \,  \right) \, \dif{t}\label{eq38}.
\end{align}
But, if we define the function $\psi$:
\begin{equation*}
\psi(\vect{x}) = \int_{0}^{x_p}   (\partial_p\varphi \circ F_t)(\xi_p,\vect{x}_{\neq p})\,\dif{\xi_p}.
\end{equation*}
we can check that $\psi$ is continuous with respect to $x_p$, piecewise continuously differentiable, that $\partial_p \psi = \partial_p \varphi \circ F_t$, and as long as $x_{p+1}$ belongs to $[-1,0]$, we have $\partial_p F_{t,p} = 1$ and:
\begin{align*}
\sup_{\vect{x}: x_p \leq 0} \psi(\vect{x}) & = \sup_{\vect{x}: x_p \leq 0} \norm{ \int_{0}^{x_p}   (\partial_p\varphi \circ F_t)(\xi_p,\vect{x}_{\neq p})\,\dif{\xi_p}}\\
& = \sup_{\vect{x}: x_p \leq 0} \norm{\int_{0}^{x_p}   \frac{\partial_p (\varphi \circ F_t)(\xi_p,\vect{x}_{\neq p})}{\partial_p F_{t,p} }\,\dif{\xi_p}}\\
& =  \sup_{\vect{x}: x_p \leq 0} \norm{ \int_{0}^{x_p}   \partial_p (\varphi \circ F_t)(\xi_p,\vect{x}_{\neq p})\,\dif{\xi_p}}\\
& \leq 2 \norm{\varphi\circ F_t}_{\infty} \leq 2  \norm{\varphi}_{\infty}.
\end{align*}
With Proposition \ref{propA4} and Corollary \ref{propA5}, the definition of $\psi$ implies that:
\begin{align*}
\mu\left(\,\charf{[-1,0]}(x_{p+1})\,\partial_p \varphi \circ F_t \,  \right) & = \mu\left(\,\charf{[-1,0]}(x_{p+1})\,\partial_p \psi \,  \right)\leq 2 \norm{\varphi}_{\infty} \dnorm{\mu},
\end{align*}
and we conclude by inserting this inequality into equation (\ref{eq38}):
\begin{align*}
&\mu\Big(\charf{[-1,0]}(x_{p+1}) \, \varphi \circ\Phi_{\epsilon}\Big)  - \mu\Big(\charf{[-1,0]}(x_{p+1}) \, \varphi \circ\Phi_{0}^{(p)}\Big)\leq 2 \,\epsilon \;\norm{\varphi}_{\infty} \dnorm{\mu} .
\end{align*}
\end{proof}

This estimate allows us to control the difference between the original dynamic $T$ and $T_{0}^{(p)} = \Phi_{0}^{(p)} \circ \tau^{  \ent}$ when the site $p+1$ stays negative:
\begin{proposition}\label{prop10}
For any measure $\mu \in \mathcal{B}(X)$, $p \in \ent$ and $m \in \nat$, we have:
\begin{align*}
\norm{\Big( T \charf{[-1,0]}(x_{p+1})\Big)^m \mu - \Big(T_{0}^{(p)} \charf{[-1,0]}(x_{p+1})\Big)^m \mu} \leq    \tfrac{1 + m \,D_0}{1 - \lambda_0} \,2\,\epsilon\, \dnorm{\mu}.
\end{align*}
 \end{proposition}
\begin{proof}
Once again, we follow the proof of Theorem 6 in \cite{KeLi05}. We define $\bar{T} = T \charf{[-1,0]}(x_{p+1})$, and $\bar{T}_{0}^{(p)} = T_{0}^{(p)} \charf{[-1,0]}(x_{p+1})$. Then, with the help of a simple telescopic sum, we have:
\begin{align}\label{eq43}
& \bar{T}_{0}^{(p)\, m} \mu  = \bar{T}^m \mu +  \sum_{k = 0}^{m-1}  \bar{T}_{0}^{(p)\,m-k-1}  \Bigg(\bar{T}_{0}^{(p)} - \bar{T} \Bigg)\bar{T}^{k} \mu.
\end{align}
Then, taking the total variation norm of this expansion and applying Proposition \ref{prop9} to control the difference between  $\Phi_{0}^{(p)}$ and $\Phi_{\epsilon} $, we get:
\begin{align}
& \norm{\bar{T}^m \mu - \bar{T}_{0}^{(p)\, m} \mu} \leq \sum_{k = 1}^{m}  \norm{\bar{T}_{0}^{(p)\,m-k-1}\,\Bigg(\bar{T}_{0}^{(p)} - \bar{T} \Bigg) \bar{T}^{k-1} \mu}\nonumber\\ 
& \qquad\leq  \sum_{k = 1}^{m}  \norm{\left(\bar{T}_{0}^{(p)} - \bar{T} \right) \bar{T}^{k-1} \mu} \leq  \sum_{k = 1}^{m}  \norm{\left(\Phi_{0}^{(p)} - \Phi_{\epsilon}\right) \tau^{\ent}\charf{[-1,0]}(x_{p+1}) \bar{T}^{k-1} \mu}\nonumber\\
& \qquad\leq \sum_{k = 1}^{m}  \norm{\left(\Phi_{0}^{(p)} - \Phi_{\epsilon}\right) \charf{[-1,0]}(x_{p+1})\tau^{\ent} \bar{T}^{k-1} \mu}\nonumber\\
&\qquad\leq   \sum_{k = 1}^{m}  \norm{\left(\Phi_{0}^{(p)} \charf{[-1,0]}(x_{p+1}) - \Phi_{\epsilon} \charf{[-1,0]}(x_{p+1})\right)\tau^{  \ent} \bar{T}^{k-1} \mu}\nonumber\\
&\qquad\leq  \sum_{k = 1}^{m} \, 2\, \epsilon  \,\dnorm{\tau^{  \ent}  \bar{T}^{k-1} \mu}\label{eq44}.
\end{align}
But $\bar{T}$ satisfies a Lasota-Yorke inequality, because of Corollary \ref{propA5}:
\begin{align}
\dnorm{\bar{T}\mu} \leq \lambda_0 \dnorm{\charf{[-1,0]}(x_{p+1})\mu} + D_0 \norm{\charf{[-1,0]}(x_{p+1})\mu} \leq \lambda_0 \dnorm{\mu} + D_0 \norm{\mu},
\label{eq44.1}\end{align}
and $\tau^{  \ent}$ also satisfies the same inequality, as a consequence of (\ref{eq0.1}). Therefore:
\begin{equation*}
\dnorm{\tau^{  \ent} \bar{T}^{k-1} \mu} \leq \left(\lambda_0^{k}  + \tfrac{D_0}{1-\lambda_0}\right) \dnorm{\mu},
\end{equation*}
and inequality (\ref{eq44}) becomes:
\begin{align*}
 \norm{\bar{T}^m \mu - \bar{T}_{0}^{(p)\, m} \mu}& \leq   \sum_{k = 1}^{m} \, 2\, \epsilon  \, \left(\lambda_0^{k}  + \tfrac{D_0}{1-\lambda_0}\right) \dnorm{\mu} \leq     \frac{1 + m \,D_0}{1 - \lambda_0} \,2\,\epsilon\, \dnorm{\mu}.
\end{align*}
\end{proof}

We are now ready to prove that if $\mu$ is a signed measure of bounded variation concentrated on $x_p \in [-1,0]$ and $x_{p+1} \in [-1,0]$, the operator $\tilde{T}^{qp}$ acting on $\mu$ is a contraction:
\begin{theorem}\label{prop11}For any measure $\mu$ in $\mathcal{B}(X)$, and any $q$ and $p$ in $\ent$, we have
\begin{equation*}
\dnorm{\tilde{T}^{qp} \charf{[-1,0]}(x_p) \charf{[-1,0]}(x_{p+1})\mu} \leq {\sigma_1} \dnorm{\mu},
\end{equation*}
where $\sigma_1$ is given by:
\begin{equation*}
\sigma_1 = \left[2\,\tfrac{D_0}{1-\lambda_0}\,\left(\lambda_0^{\gamma - n} + \lambda_0^{\gamma}  + c \varsigma^n\right)\,+  \left(\tfrac{D_0}{1-\lambda_0}\right)^2\ \tfrac{1 + n \,D_0}{1 - \lambda_0} \,\left(4\,\epsilon +  \norm{E_{\epsilon}}\right)\,\right].
\end{equation*}
\end{theorem}
\begin{proof}Remember that $\tilde{T}^{qp}$ was defined in (\ref{eq33.-1}) as $T^{\gamma} \proj{qp}$. If we apply the Lasota-Yorke inequality (\ref{eq17.2}) to $T^{\gamma - n}$, for some strictly positive $n < \gamma$, we have:
\begin{align}
& \dnorm{T^{\gamma} \proj{qp} \charf{[-1,0]}(x_p) \charf{[-1,0]}(x_{p+1})\mu}\nonumber\\
\begin{split}
& \qquad \leq \lambda_0^{\gamma - n} \dnorm{T^{n} \proj{qp} \charf{[-1,0]}(x_p) \charf{[-1,0]}(x_{p+1})\mu}\\
& \qquad\qquad + \tfrac{D_0}{1 - \lambda_0}\, \norm{T^{n} \proj{qp} \charf{[-1,0]}(x_p) \charf{[-1,0]}(x_{p+1})\mu}\label{eq45} .
\end{split}
\end{align}
Applying once again the Lasota-Yorke to the first term of this inequality, using Lemma \ref{prop8} and Corollary \ref{propA5}, we get: 
\begin{align*}
& \dnorm{T^{n} \proj{qp} \charf{[-1,0]}(x_p) \charf{[-1,0]}(x_{p+1})\mu}\\
& \qquad \leq \lambda_0^{n} \dnorm{\proj{qp} \charf{[-1,0]}(x_p) \charf{[-1,0]}(x_{p+1})\mu} + \tfrac{D_0}{1 - \lambda_0} \norm{\proj{qp} \charf{[-1,0]}(x_p) \charf{[-1,0]}(x_{p+1})\mu}\\
& \qquad\leq 2\,\tfrac{D_0}{1-\lambda_0}\,\lambda_0^{n} \,\dnorm{\mu} + 2\,\tfrac{D_0}{1 - \lambda_0} \norm{\mu}\\
& \qquad\leq 2\,\tfrac{D_0}{1-\lambda_0}\,\left(1+ \lambda_0^{n}  \right)\, \dnorm{\mu}.
\end{align*}
And so, inequality (\ref{eq45}) becomes:
\begin{align}
& \dnorm{\tilde{T}^{qp} \charf{[-1,0]}(x_p) \charf{[-1,0]}(x_{p+1})\mu} \nonumber\\
\begin{split}
& \qquad \leq 2\,\tfrac{D_0}{1-\lambda_0}\,\left(\lambda_0^{\gamma - n} + \lambda_0^{\gamma}  \right)\, \dnorm{\mu} + \tfrac{D_0}{1 - \lambda_0}\, \norm{T^{n} \proj{qp} \charf{[-1,0]}(x_p) \charf{[-1,0]}(x_{p+1})\mu}\label{eq45.0} .
\end{split}
\end{align}

For the second term in (\ref{eq45.0}), we first note that, by the definition of $\proj{qp}$ in (\ref{eq28}) and the fact that $\hinv{-}$ is concentrated on $[-1,0]$, we have, for any $s \in \ent$ and any measure $\nu$:
\begin{equation}
\proj{qp}\charf{[-1,0]}(x_{s})\nu = \charf{[-1,0]}(x_{s})\proj{qp}\charf{[-1,0]}(x_{s})\nu \label{eq45.0.1}.
\end{equation}
We can therefore introduce an operator $\charf{[-1,0]}(x_{p+1})$ in front of $\proj{qp}$ in the second term of (\ref{eq45.0}):
\begin{equation*}\begin{split}
& \norm{T^{n} \proj{qp} \charf{[-1,0]}(x_p) \charf{[-1,0]}(x_{p+1})\mu}\\
& \qquad =  \norm{T^{n} \charf{[-1,0]}(x_{p+1})\proj{qp} \charf{[-1,0]}(x_p) \charf{[-1,0]}(x_{p+1})\mu}.\end{split}
\end{equation*}
Since $x_{p+1}$ is initially negative, either it stays negative up to time $n$ or there is a sign flip at some intermediate time $k$.
\begin{align}
& \norm{T^{n} \proj{qp} \charf{[-1,0]}(x_p) \charf{[-1,0]}(x_{p+1})\mu}\nonumber\\
\begin{split}&\leq \norm{\left(T \charf{[-1,0]}(x_{p+1})\right)^{n} \proj{qp} \charf{[-1,0]}(x_p) \charf{[-1,0]}(x_{p+1})\mu}\phantom{\sum^{n - 1}}\\
&\quad+ \sum_{k = 1}^{n - 1} \norm{ T^{n - k} \charf{(0,1]}(x_{p+1}) \left(T \charf{[-1,0]}(x_{p+1})\right)^k \proj{qp} \charf{[-1,0]}(x_p) \charf{[-1,0]}(x_{p+1})\mu}\label{eq45.1}.\end{split}
\end{align}
We can then apply Proposition \ref{prop10} to the first term of (\ref{eq45.1}), and replace the initial dynamic by $T_{0}^{(p)} $ up to an error that grows at most linearly with time. This, together with  Lemma \ref{prop8} and Corollary \ref{propA5}, leads to:

\begin{align}
&\norm{\left(T \charf{[-1,0]}(x_{p+1})\right)^{n} \proj{qp} \charf{[-1,0]}(x_p) \charf{[-1,0]}(x_{p+1})\mu}\nonumber \\
& \quad\leq \norm{\Big(T_{0}^{(p)} \charf{[-1,0]}(x_{p+1})\Big)^n \proj{qp} \charf{[-1,0]}(x_p) \charf{[-1,0]}(x_{p+1})\mu}\nonumber\\
& \quad\qquad +  2\, \frac{1 + n \,D_0}{1 - \lambda_0} \,\epsilon\, \dnorm{\proj{qp} \charf{[-1,0]}(x_p) \charf{[-1,0]}(x_{p+1})\mu}\nonumber\\
\begin{split}
& \quad\leq \norm{\Big(T_{0}^{(p)} \charf{[-1,0]}(x_{p+1})\Big)^n \proj{qp} \charf{[-1,0]}(x_p) \charf{[-1,0]}(x_{p+1})\mu}\\
& \quad\qquad + 4\,\frac{D_0}{1-\lambda_0}\ \frac{1 + n \,D_0}{1 - \lambda_0} \,\epsilon\, \dnorm{\mu}.
\label{eq47}
\end{split}
\end{align}

Now that we are left with the decoupled dynamic at site $p$, we can take advantage of the mixing properties of the local dynamic $\tau$ as in \cite{KeLi06}. Indeed, for any measure $\nu$, we see that 
\begin{align*}
& \Big(T_{0}^{(p)} \charf{[-1,0]}(x_{p+1})\Big)^n \charf{[-1,0]}(x_p)\nu(\varphi)\\
& = \nu\left(\,\charf{[-1,0]}(x_p)\, \varphi(T_{0}^{(p)\,n} \vect{x})\ \prod_{t = 0}^{n} \charf{[-1,0]}(T_0^{(p)\,t} x_{p+1})\,\right).
\end{align*}
Here, $\charf{[-1,0]}(T_0^{(p)\,t} x_{p+1})$ does not depend on $x_p$, but only on the variables $x_{>p}$. Moreover, the sign of $x_p$ is initially fixed to be negative, therefore, $T_{0}^{(p)\,n} x_p = \tau^n x_p$ and $T_{0}^{(p)\,n} x_{q}$ for $q \neq p$ depends only on the sign of $x_p$ which is fixed and negative. So, the dynamic is actually the product of two dynamics, $\tau^n$ acting on $x_p$, and $T^n$ acting on $\vect{x}_{\neq p}$ with fixed negative boundary conditions in $x_p$. If we define $T_{\neq p} x_q = T (\vect{x}_{\neq p}, -1)_q$, and remember that $\tau$ preserves the signs, we see that:
\begin{align*}
& \Big(T_{0}^{(p)} \charf{[-1,0]}(x_{p+1})\Big)^n \charf{[-1,0]}(x_p)\nu(\varphi)\\
& = \nu\left(\,\charf{[-1,0]}(\tau^n x_p)\, \varphi(\tau^n x_p , T_{\neq p}^{n} \vect{x}_{\neq p})\ \prod_{t = 0}^{n} \charf{[-1,0]}(T^{t} x_{p+1})\,\right)\\
& \leq \norm{\left(\tau^n \times \charf{\neq p}\right)\nu}\, \norm{\varphi}_{\infty}.
\end{align*}
If we apply this inequality to the first term of (\ref{eq47}), together with (\ref{eq45.0.1}) to create a $\charf{[-1,0]}(x_p)$ in front of $\proj{qp}$, we find:
\begin{align*}
& \norm{\Big(T_{0}^{(p)} \charf{[-1,0]}(x_{p+1})\Big)^n\proj{qp} \charf{[-1,0]}(x_p) \charf{[-1,0]}(x_{p+1})\mu}\\
& \qquad= \norm{\Big(T_{0}^{(p)} \charf{[-1,0]}(x_{p+1})\Big)^n \charf{[-1,0]}(x_p)\proj{qp} \charf{[-1,0]}(x_p) \charf{[-1,0]}(x_{p+1})\mu}\\
&  \qquad\leq \norm{\left(\tau^n \times \charf{\neq p} \right)\proj{qp} \charf{[-1,0]}(x_p) \charf{[-1,0]}(x_{p+1})\mu}.
\end{align*}
But, if $\varphi$ is some continuous function, with $\norm{\varphi}_{\infty} \leq 1$, and if we define $\psi$:
\begin{equation*}
\psi(\vect{x}) =  \charf{[-1,0]}(x_{p+1}) \int \varphi(\xi_{\prec_q p},x_{\succeq_q p}) \prod_{s \prec_q p } \hinv{-}(\xi_s)\ \dif \xi_s ,
\end{equation*}
we see that:
\begin{align*}
& \Bigl(\left(\tau^n \times \charf{\neq p} \right)\proj{qp} \charf{[-1,0]}(x_p) \charf{[-1,0]}(x_{p+1})\mu\Bigr)(\varphi)\\
&\quad = \mu \left(\charf{[-1,0]}(x_p) \left(\psi(\tau^n x_p,\vect{x}_{\neq p}) - \int  \psi(\tau^n \xi_p,\vect{x}_{\neq p}) \hinv{-}(\xi_p)\ \mathrm{d}\xi_p\right) \right)\\
& \quad= \mu \left( \charf{[-1,0]}(x_p) \partial_p \left(  \int_{-1}^{x_p}\psi(\tau^n \xi_p,\vect{x}_{\neq p}) \dif \xi_p - (x_p + 1) \int  \psi(\xi_p,\vect{x}_{\neq p}) \hinv{-}(\xi_p)\ \mathrm{d}\xi_p\right) \right).
\end{align*}
And now, by definition of the bounded variation norm, inequality (\ref{eq0.2}), and the fact that $\norm{\psi}_{\infty} \leq \norm{\varphi}_{\infty} \leq 1$, we have:
\begin{align*}
& \Bigl(\left(\tau^n \times \charf{\neq p} \right)\proj{qp} \charf{[-1,0]}(x_p) \charf{[-1,0]}(x_{p+1})\mu\Bigr)(\varphi)\\
& \quad\leq \dnorm{\mu} \sup_{x_p \in [-1,0]} \norm{ \int_{-1}^{1}\psi(\tau^n \xi_p,\vect{x}_{\neq p})\ \charf{[-1,x_p]}(\xi_p) \ \dif \xi_p - (x_p + 1) \int  \psi(\xi_p,\vect{x}_{\neq p}) \hinv{-}(\xi_p)\ \mathrm{d}\xi_p}_{\infty}\\
&\quad \leq \dnorm{\mu} \sup_{x_p \in [-1,0]} c\, \varsigma^n \dnorm{\charf{[-1,x_p]}}_{BV} \leq 2\, c \, \varsigma^n \dnorm{\mu}.
\end{align*}

We can now go back to (\ref{eq47}). Indeed, we just proved that:
\begin{align*}
& \norm{\Big(T_{0}^{(p)} \charf{[-1,0]}(x_{p+1})\Big)^n\proj{qp} \charf{[-1,0]}(x_p) \charf{[-1,0]}(x_{p+1})\mu}\leq  2\, c \, \varsigma^n \dnorm{\mu}.
\end{align*}
If we insert this bound into (\ref{eq47}), we have:
\begin{align}
&\norm{\left(T \charf{[-1,0]}(x_{p+1})\right)^{n} \proj{qp} \charf{[-1,0]}(x_p) \charf{[-1,0]}(x_{p+1})\mu}\nonumber \\
& \quad\leq \left(2\, c\, \varsigma^n + 4\,\frac{D_0}{1-\lambda_0}\ \frac{1 + n \,D_0}{1 - \lambda_0} \,\epsilon\right)\,\dnorm{\mu}.
\end{align}
And consequently (\ref{eq45.1}) can be rewritten as:
\begin{align}
& \norm{T^{n} \proj{qp} \charf{[-1,0]}(x_p) \charf{[-1,0]}(x_{p+1})\mu}\nonumber\\
\begin{split}&\quad\leq \left(2\, c\, \varsigma^n + 4\,\frac{D_0}{1-\lambda_0}\ \frac{1 + n \,D_0}{1 - \lambda_0} \,\epsilon\right)\,\dnorm{\mu}\phantom{\sum^{n - 1}}\\
&\quad+ \sum_{k = 1}^{n - 1} \norm{ T^{n - k} \charf{(0,1]}(x_{p+1}) \left(T \charf{[-1,0]}(x_{p+1})\right)^k \proj{qp} \charf{[-1,0]}(x_p) \charf{[-1,0]}(x_{p+1})\mu}\label{eq45.1.2}.\end{split}
\end{align}

We are then left with the cases where a sign flip happens at some time $k$. Since we know that at time $k-1$, $x_{p+1}$ belongs to $[-1,0]$, and at time $k$, $x_{p+1}$ belongs to $(0,1]$, $x_{p+1}$ at time $k-1$ has to belong to the small set $E_{\epsilon}$. Therefore, applying (\ref{eq8}) with $\Omega = \varnothing$ and $\Lambda = \{p+1\}$, we get:
\begin{align}
& \norm{ T^{n - k} \charf{(0,1]}(x_{p+1}) \left(T \charf{[-1,0]}(x_{p+1})\right)^k \proj{qp} \charf{[-1,0]}(x_p) \charf{[-1,0]}(x_{p+1})\mu}\nonumber\\
& \quad\leq  \norm{ \charf{E_{\epsilon}}(x_{p+1})\left(T \charf{[-1,0]}(x_{p+1})\right)^{k-1} \proj{qp} \charf{[-1,0]}(x_p) \charf{[-1,0]}(x_{p+1})\mu}\nonumber\\
& \quad\leq \frac{\norm{E_{\epsilon}}}{2} \dnorm{\left(T \charf{[-1,0]}(x_{p+1})\right)^{k-1} \proj{qp} \charf{[-1,0]}(x_p) \charf{[-1,0]}(x_{p+1})\mu}.\label{eq48}
\end{align}
But from (\ref{eq44.1}), we can check that $\lVert\,\left(T \charf{[-1,0]}(x_{p+1})\right)^{k-1} \nu\,\rVert \leq (\lambda_0^{k-1}  + \tfrac{D_0}{1-\lambda_0}) \dnorm{\nu}$. This inequality and Lemma \ref{prop8} applied to (\ref{eq48}) implies:
\begin{align*}
& \norm{ T^{n - k} \charf{(0,1]}(x_{p+1}) \left(T \charf{[-1,0]}(x_{p+1})\right)^k \proj{qp} \charf{[-1,0]}(x_p) \charf{[-1,0]}(x_{p+1})\mu}\nonumber\\
& \qquad\leq \frac{\norm{E_{\epsilon}}}{2} \left(\lambda_0^{k-1} + \frac{D_0}{1-\lambda_0}\right)\,\dnorm{\proj{qp} \charf{[-1,0]}(x_p) \charf{[-1,0]}(x_{p+1})\mu}\\
& \qquad\leq \norm{E_{\epsilon}} \left(\lambda_0^{k-1} + \frac{D_0}{1-\lambda_0}\right)\,\frac{D_0}{1-\lambda_0}\,\dnorm{\mu}.
\end{align*}
We insert this inequality into (\ref{eq45.1.2}), take the geometric series as an upper bound on the sum, and we get:
\begin{align*}
& \norm{T^{n} \proj{qp} \charf{[-1,0]}(x_p) \charf{[-1,0]}(x_{p+1})\mu}\nonumber\\
&\leq \left(2\, c\, \varsigma^n + 4\,\tfrac{D_0}{1-\lambda_0}\ \tfrac{1 + n \,D_0}{1 - \lambda_0} \,\epsilon\right)\,\dnorm{\mu}+ \sum_{k = 1}^{n - 1} \norm{E_{\epsilon}} \left(\lambda_0^{k-1} + \tfrac{D_0}{1-\lambda_0}\right)\,\tfrac{D_0}{1-\lambda_0}\,\dnorm{\mu}\\
& \leq \left(2\, c\, \varsigma^n + \tfrac{D_0}{1-\lambda_0}\ \tfrac{1 + n \,D_0}{1 - \lambda_0} \,\left(4\,\epsilon +  \norm{E_{\epsilon}}\right)\right)\,\dnorm{\mu}.
\end{align*}
And finally, we insert this inequality in (\ref{eq45.0}):
\begin{align*}
& \dnorm{\tilde{T}^{qp} \charf{[-1,0]}(x_p) \charf{[-1,0]}(x_{p+1})\mu} \\
& \leq 2\,\tfrac{D_0}{1-\lambda_0}\,\left(\lambda_0^{\gamma - n} + \lambda_0^{\gamma}  + c \varsigma^n\right)\,\dnorm{\mu}+  \left(\tfrac{D_0}{1-\lambda_0}\right)^2\ \tfrac{1 + n \,D_0}{1 - \lambda_0} \,\left(4\,\epsilon +  \norm{E_{\epsilon}}\right)\,\dnorm{\mu}\\
& \leq \left[2\,\tfrac{D_0}{1-\lambda_0}\,\left(\lambda_0^{\gamma - n} + \lambda_0^{\gamma}  + c \varsigma^n\right)\,+  \left(\tfrac{D_0}{1-\lambda_0}\right)^2\ \tfrac{1 + n \,D_0}{1 - \lambda_0} \,\left(4\,\epsilon +  \norm{E_{\epsilon}}\right)\,\right]\,\dnorm{\mu},
\end{align*}
and by definition of $\sigma_1$, we therefore proved that:
\begin{align*}
& \dnorm{\tilde{T}^{qp} \charf{[-1,0]}(x_p) \charf{[-1,0]}(x_{p+1})\mu} \leq \sigma_1\,\dnorm{\mu}.
\end{align*}
\end{proof}

\subsection{Polymer expansion}
\label{contexp}
We are now at a turning point of our reasoning. Indeed, the Peierls argument from Section \ref{peierls} tells us that the probability of having positive sites is small with respect to some class of initial measures, and Theorem \ref{prop11} allows us to control the dynamic restricted to the negative phase. Combining these two arguments, a contour estimate and a decoupling estimate, is usually called a polymer expansion in Statistical Physics, and we will see that it implies the exponential convergence to equilibrium for a wide class of initial measures.

\begin{theorem}\label{prop12}Assume that $\tau$ belongs to $\mathcal{T}(D_0,c,\varsigma)$ for $D_0, c > 0$ and $\varsigma \in (0,1)$ given and that $\kappa$ is larger than some $\kappa_1$ that depends on $D_0$, $c$ and $\varsigma$. Then, there is some $\epsilon_1 \in (0,1)$ such that, if $\epsilon \in [0,\epsilon_1]$, there is some $\sigma < 1$ such that for any $K >0$ there is some constant $C >0$ such that:
\begin{equation*}
 \norm{T^t \mu(\varphi)} \leq C \norm{\Lambda} \sigma^t \dnorm{\mu} \norm{\varphi}_{\infty}.
\end{equation*}
 for any signed measure of zero mass $\mu$ in $\mathcal{B}(K,3 \alpha_0,\theta_0)$ with $\alpha_0$ defined in (\ref{eq3.3}) and $\theta_0$ defined in (\ref{eq3.3bis}), and for any continuous function $\varphi$ depending only on the variables in $\Lambda \subset \ent$.
\end{theorem}
\begin{proof}
Assume for the beginning that $t = m \gamma$ for some $m \in \nat$, and let $t_k = k \gamma$. The expansion of (\ref{eq33}) gives us:
\begin{equation}\label{eq57.1}
\norm{T^t \mu (\varphi)}\leq 2\,\sum_{\vect{p} \in \mathcal{F}(\Lambda)} \norm{\tilde{T}^{p_{m-2} p_{m-1}}\ldots \tilde{T}^{0 p_{0}}\mu}\,\norm{\varphi}_{\infty}.
\end{equation}
We define  and, for any $\vect{p} \in \mathcal{F}(\Lambda)$:
\begin{equation*}
G(\vect{p}) = \{(p_0,t_{0}), (p_{0} +1,t_{0}), \ldots, (p_{m-1},t_{m-1}), (p_{m-1} +1,t_{m-1}) \}.
\end{equation*}
If $\mathcal{P}(G(\vect{p}))$ is the set of subsets of $G(\vect{p})$, then, for any $\Omega \in \mathcal{P}(G(\vect{p}))$ and for any $t$, we can define $\Omega_t = \{q \, | \,(q,t) \in \Omega\}$ , the projection of $\Omega$ on time $t$. Then, we introduce the set $F(\Omega,t)$, defined by:
\begin{equation}
\charf{F(\Omega,t)} = \charf{(0,1]^{\Omega_t}} \,\prod_{q\,:\, (q,t) \in G(\vect{p}) \setminus \Omega} \charf{[-1,0]}(x_q)\label{eq57}.
\end{equation}
If we sum over all possibilities, (\ref{eq57.1}) can be rewritten as:
\begin{equation}\label{eq58}
\norm{T^t \mu (\varphi)}\leq 2\,\sum_{\vect{p} \in \mathcal{F}(\Lambda)}\, \sum_{\Omega \in \mathcal{P}(G(\vect{p}))} \norm{\tilde{T}^{p_{m-2} p_{m-1}}\charf{F(\Omega,t_{m-1})}\ldots \tilde{T}^{0 p_{0}}\charf{F(\Omega,0)}\mu}\,\norm{\varphi}_{\infty}.
\end{equation}
We note that the number of terms in the sums grows at most exponentially with $t$. Indeed, $\norm{\mathcal{F}(\Lambda)} \leq \norm{\Lambda}\, \gamma^m$ and $\norm{\mathcal{P}(G(\vect{p}))} \leq 2^{2 m}$, so
\begin{equation}\label{eq58.1}
\norm{\mathcal{F}(\Lambda)}\,\norm{\mathcal{P}(G(\vect{p}))} \leq \norm{\Lambda} \,(4 \gamma)^{\frac{t}{\gamma}}.
\end{equation}

Consider now $\mathcal{P}'$, the set of the $\Omega \in \mathcal{P}(G(\vect{p}))$ such that at least $\frac{m}{2}$ of the $m$ sets $\Omega_{t_k}$ are empty. If $\Omega$ belongs to $\mathcal{P}'$, we know that at least $\frac{m}{2}$ of the $m$ operators $\tilde{T}^{p_{k} p_{k+1}}\charf{F(\Omega,t_{k})}$ are bounded in bounded variation norm by Theorem \ref{prop11}:
\begin{equation*}\Omega_{t_k} = \varnothing \Rightarrow \dnorm{\tilde{T}^{p_{m-2} p_{m-1}}\charf{F(\Omega,t_{m-1})} \nu} \leq {\sigma_1} \dnorm{\nu},
\end{equation*}
and the other operators $\tilde{T}^{p_{k} p_{k+1}}\charf{F(\Omega,t_{k})}$ are bounded by (\ref{eq17.2}), Lemma \ref{prop8} and Corollary \ref{propA5}:
\begin{equation*}\Omega_{t_k} \neq \varnothing \Rightarrow \dnorm{\tilde{T}^{p_{m-2} p_{m-1}}\charf{F(\Omega,t_{m-1})} \nu} \leq 2\  \left(\tfrac{D_0}{1- \lambda_0}\right) \ \left(\lambda_0 +\tfrac{D_0}{1-\lambda_0}\right) \dnorm{\nu}.
\end{equation*}
Therefore, for any $\Omega \in \mathcal{P}'$, using the fact that $\mu \in \mathcal{B}(K,3\alpha_0,\theta_0)$, we have:
\begin{align}
& \norm{\tilde{T}^{p_{m-2} p_{m-1}}\charf{F(\Omega,t_{m-1})}\ldots \tilde{T}^{0 p_{0}}\charf{F(\Omega,0)}\mu}\nonumber\\
&\qquad \leq \left(2\  \left(\tfrac{D_0}{1- \lambda_0}\right) \ \left(\lambda_0 +\tfrac{D_0}{1-\lambda_0}\right)\right)^{\frac{m}{2}}\, \sigma_{1}^{\frac{m}{2}}\,\dnorm{\mu}\nonumber\\
&\qquad \leq \left[\left(2\  \left(\tfrac{D_0}{1- \lambda_0}\right) \ \left(\lambda_0 +\tfrac{D_0}{1-\lambda_0}\right)\right)\, \sigma_{1}\right]^{\frac{m}{2}}\,K \theta_0\label{eq59}.
\end{align}

If $\Omega$ does not belong to $\mathcal{P}'$, we know that we have at least  $ \frac{m}{2}$ characteristics functions of $(0,1]$, and we will use the Peierls argument to show that this only happens with small probability. We start by defining the sequence of measures $\mu_k$ by:
\begin{equation}\label{eq60}
\begin{cases}
& \mu_0 = \charf{F(\Omega,0)}\mu\\
& \mu_k =  \charf{F(\Omega,t_k)}\tilde{T}^{p_{k-2} p_{k-1}}\charf{F(\Omega,0)}\mu_{k -1} 
\end{cases}\end{equation}
Then, using Lemma \ref{prop8}, we can see that:
\begin{equation}
\norm{\tilde{T}^{p_{m-2} p_{m-1}}\charf{F(\Omega,t_{m-1})}\ldots \tilde{T}^{0 p_{0}}\charf{F(\Omega,0)}\mu} = \norm{\tilde{T}^{p_{m-2} p_{m-1}}\mu_{m-1}} \leq 2 \tnorm{\mu_{m-1}}{\theta_0}
\end{equation}
Assume that we already picked up some $\Lambda_k \subset \ent$ and consider $\vvvert \charf{(0,1]^{\Lambda_k}} \mu_k\vvvert_{\theta_0}$. By definition of $\mu_k$, and since $F(\Omega,t_k) \subset (0,1]^{\Omega_{t_k}}$, we get:
\begin{align*}
\tnorm{\charf{(0,1]^{\Lambda_k}} \mu_k}{\theta_0} & \leq \tnorm{\charf{(0,1]^{\Lambda_k}}\charf{F(\Omega,t_k)} \tilde{T}^{p_{k-2} p_{k-1}}  \mu_{k-1}}{\theta_0}\\
& \leq \tnorm{\charf{(0,1]^{\Lambda_k \cup \Omega_{t_k}}}\tilde{T}^{p_{k-2} p_{k-1}}\mu_{k-1}}{\theta_0}
\end{align*}
We can now apply inequality (\ref{eq25.-1}) to the measure $\mu_{k-1}$:
\begin{align}
& \tnorm{\charf{(0,1]^{\Lambda_k}}\mu_k}{\theta_0}\leq \sum_{\Gamma^{\scriptscriptstyle(k)} \in \mathcal{G}(\Lambda_k \cup \Omega_{t_k})} {\alpha_0}^{\norm{\partial \Gamma^{\scriptscriptstyle(k)}} - \norm{\partial \Gamma^{\scriptscriptstyle(k)}_{t_{k-1}}}} \Big\vvvert \charf{(0,1]^{\scriptstyle \partial \Gamma^{(k)}_{t_{k-1}}}}\proj{p_{k-1}p_{k}}\mu_{k-1}\Big\vvvert_{\theta_0}\nonumber
\end{align}
If we define $\Lambda_{k-1} = \{ q\, |\, (q,t_{k-1}) \in \Gamma^{\scriptscriptstyle(k)}\}$, we then find:
\begin{align}
& \tnorm{\charf{(0,1]^{\Lambda_k}}\mu_k}{\theta_0} \leq \sum_{\Gamma^{\scriptscriptstyle(k)} \in \mathcal{G}(\Lambda_k \cup \Omega_{t_k})}{\alpha_0}^{\norm{\partial\Gamma^{\scriptscriptstyle(k)}} - \norm{\Lambda_{k-1}}} \, \tnorm{\charf{(0,1]^{\Lambda_{k-1}}}\proj{p_{k-1}p_{k}}\mu_{k-1}}{\theta_0}\nonumber
\end{align}
Now, remember that the operator $\proj{qp}$ only integrates over $[-1,0]$. Therefore:
\begin{equation*}\charf{(0,1]^{\Lambda_{k-1}}}\proj{p_{k-1}p_{k}} =  \charf{(0,1]^{\Lambda_{k-1}}}\proj{p_{k-1}p_{k}} \charf{(0,1]^{\Lambda_{k-1}}},
\end{equation*}
and using Corollary \ref{propA5} and Lemma \ref{prop8}, we get:
\begin{align}
 \tnorm{\charf{(0,1]^{\Lambda_k}}\mu_k}{\theta_0} & \leq \sum_{\Gamma^{\scriptscriptstyle(k)} \in \mathcal{G}(\Lambda_k \cup \Omega_{t_k})}{\alpha_0}^{\norm{\partial\Gamma^{\scriptscriptstyle(k)}} - \norm{\Lambda_{k-1}}} \, \tnorm{\proj{p_{k-1}p_{k}}\charf{(0,1]^{\Lambda_{k-1}}}\mu_{k-1}}{\theta_0}\nonumber\\
&\leq 2\, \sum_{\Gamma^{\scriptscriptstyle(k)} \in \mathcal{G}(\Lambda_k \cup \Omega_{t_k})}{\alpha_0}^{\norm{\partial\Gamma^{\scriptscriptstyle(k)}} - \norm{\Lambda_{k-1}}} \, \tnorm{\charf{(0,1]^{\Lambda_{k-1}}}\mu_{k-1}}{\theta_0}\label{eq61}.
\end{align}
We can iterate this inequality to find an upper bound on $\tnorm{\mu_{m-1}}{\theta_0}$. Starting from $\Lambda_{m-1} = \emptyset$:
\begin{align}
& \tnorm{\mu_{m-1}}{\theta_0} \leq 2\, \sum_{\Gamma^{\scriptscriptstyle(m-1)} \in \mathcal{G}(\Omega_{t_{m-1}})}{\alpha_0}^{\norm{\partial\Gamma^{\scriptscriptstyle(m-1)}} - \norm{\Lambda_{m-2}}} \, \tnorm{\charf{(0,1]^{\Lambda_{m-2}}}\mu_{m-2}}{\theta_0}\nonumber\\
\begin{split}
& \quad \leq 2^{m-1} \sum_{\Gamma^{\scriptscriptstyle(m-1)} \in \mathcal{G}(\Omega_{t_{m-1}})}\ldots\sum_{\Gamma^{\scriptscriptstyle(1)} \in \mathcal{G}(\Omega_{t_{1}}\cup \Lambda_1)}
{\alpha_0}^{\norm{\partial\Gamma^{\scriptscriptstyle(m-1)}} - \norm{\Lambda_{m-2}}}\ldots{\alpha_0}^{\norm{\partial\Gamma^{\scriptscriptstyle(1)}} - \norm{\Lambda_{0}}} \\
&  \,\qquad\qquad\qquad\qquad\qquad\qquad\qquad\qquad\qquad\qquad\qquad \tnorm{\charf{(0,1]^{\Lambda_{0}}}\mu_{0}}{\theta_0}\label{eq61.1}
\end{split}
\end{align}
But we can also see that:
\begin{equation*}
\tnorm{\charf{(0,1]^{\Lambda_{0}}}\mu_{0}}{\theta_0} = \tnorm{\charf{(0,1]^{\Lambda_{0}}}\charf{F(\Omega,0)}\mu}{\theta_0}
\leq  \tnorm{\charf{(0,1]^{\Omega_0 \cup \Lambda_0}}\mu}{\theta_0}.
\end{equation*}
For the convenience, let us define $\Gamma^{\scriptscriptstyle(0)} = \Omega_0 \cup \Lambda_0 \times \{0\}$ and $\Gamma = \bigcup_{k = 0}^{m-1} \Gamma^{\scriptscriptstyle (k)}$, and let $\sum_{\Gamma}$ denote the sum over all $\Gamma^{\scriptscriptstyle (k)}$ from the previous inequality. We can check that:
\begin{equation*}
\sum_{t = t_{k-1}+1}^{t_k} \norm{\partial \Gamma_t} = \norm{\partial\Gamma^{\scriptscriptstyle(k)}} - \norm{\Lambda_{k-1}}
\end{equation*}
And so:
\begin{equation*}
\sum_{k = 1}^{m-1} (\norm{\partial\Gamma^{\scriptscriptstyle(k)}} - \norm{\Lambda_{k-1}})= \sum_{t = 1}^{t_{m-1}} \norm{\partial \Gamma_t} = \norm{\partial \Gamma} - \norm{\partial \Gamma_0}.\end{equation*}
Since $\Omega_0 \cup \Lambda_0 = \partial \Gamma_0$, (\ref{eq61.1}) then becomes:
\begin{align}
\tnorm{\mu_{m-1}}{\theta_0} \leq 2^{m-1} \sum_{\Gamma} {\alpha_0}^{\norm{\partial\Gamma} - \norm{\partial \Gamma_0}}  \, \tnorm{\charf{(0,1]^{\partial \Gamma_0}} \mu}{\theta_0}
\end{align}
We can now use the assumption that $\mu$ belongs to $\mathcal{B}(K,3 \alpha_0,\theta_0)$, and for any $\Omega \in \mathcal{P}'$, we get:
\begin{align*}
\norm{\tilde{T}^{p_{m-2} p_{m-1}}\charf{F(\Omega,t_{m-1})}\ldots \tilde{T}^{0 p_{0}}\charf{F(\Omega,0)}\mu} & \leq 2^m \sum_{\Gamma} {\alpha_0}^{\norm{\partial\Gamma} - \norm{\partial \Gamma_0}}  \, \tnorm{\charf{(0,1]^{\partial \Gamma_0}} \mu}{\theta_0}\\
&  \leq K\,2^m \sum_{\Gamma} {\alpha_0}^{\norm{\partial\Gamma} - \norm{\partial \Gamma_0}} (3 \alpha_0)^{\norm{\partial \Gamma_0}}
\end{align*}
Consider now the outer paths associated to the cluster $\Gamma$. We can check that the number of diagonal, vertical and horizontal edges have to be equal and that if $\Gamma$ has $c$ connected parts, $n_d = n_v = n_h \geq \frac{m}{2} - c$. Since $\norm{\partial \Gamma} \geq n_h + c$ from (\ref{eq7}) still holds, we can see that, as long as $\alpha_0 < \frac{1}{81}$, we have:
\begin{align}
\norm{\tilde{T}^{p_{m-2} p_{m-1}}\charf{F(\Omega,t_{m-1})}\ldots \tilde{T}^{0 p_{0}}\charf{F(\Omega,0)}\mu} & \leq  K\,2^m \sum_{c = 1}^{\infty}\,\sum_{n_h = \frac{m}{2} - c}^{\infty}  3^{3 n_h}{(3\alpha_0)}^{n_h + c}\nonumber\\
& \leq \frac{K}{26 (1 - 81 \alpha_0)} \;2^m\; (81 \alpha_0)^{\frac{m}{2}}\nonumber\\
& \leq \frac{K}{26 (1 - 81 \alpha_0)} \; (18\,\sqrt{\alpha_0})^{m} \label{eq63}.
\end{align}

We can now move to the conclusion of this proof. We need to choose $\sigma \in (0,1)$, $\gamma$ and $n$ and the parameters of the model (namely $\kappa$ and $\epsilon$) such that:
\begin{equation}\begin{cases}
&{\displaystyle\alpha_0 < \frac{1}{81}}\\
&{\displaystyle\rule{0pt}{14pt} 18\,\sqrt{\alpha_0} < \frac{\sigma^{\gamma}}{4 \gamma}}\\
&{\displaystyle \left[\left(2\  \left(\tfrac{D_0}{1- \lambda_0}\right) \ \left(\lambda_0 +\tfrac{D_0}{1-\lambda_0}\right)\right)\, \sigma_{1}\right]^{\frac{1}{2}}< \frac{\sigma^{\gamma}}{4 \gamma}}
\end{cases}\end{equation}
We start by choosing $\sigma \in (0,1)$, $\gamma$ and $n$ independently of $\kappa$ such that:
\begin{equation*}
\left[\left(2\  \left(\tfrac{D_0}{1- \varsigma}\right) \ \left(\varsigma +\tfrac{D_0}{1-\varsigma}\right)\right) 2\,\tfrac{D_0}{1-\varsigma}\,\left(\varsigma^{\gamma - n} + \varsigma^{\gamma}  + c \varsigma^n\right)\,\right]^\frac{1}{2} \leq \frac{\sigma^{\gamma}}{4 \gamma}
\end{equation*}
We have seen that we can always choose $\varsigma$ such that $\lambda_0 \leq \varsigma$, and by definition of $\sigma_1$, it implies that:
\begin{equation*}
\lim_{\epsilon \to 0} \left[\left(2\  \left(\tfrac{D_0}{1- \lambda_0}\right) \ \left(\lambda_0 +\tfrac{D_0}{1-\lambda_0}\right)\right)\, \sigma_{1}\right]^{\frac{1}{2}} \leq  \frac{\sigma^{\gamma}}{4 \gamma}
\end{equation*}
Using (\ref{eq21.-1}), we now choose $\kappa_1 > 0$ such that, if $\kappa \geq \kappa_1$, the following inequalities are satisfied:
\begin{equation}\begin{cases}
&{\displaystyle \lim_{\epsilon \to 0} \alpha_0 < \frac{1}{81}}\\
&{\displaystyle\rule{0pt}{14pt} \lim_{\epsilon \to 0} 18\,\sqrt{\alpha_0} < \frac{\sigma^{\gamma}}{4 \gamma}}.
\end{cases}\end{equation}

Under these assumptions, there is some $\epsilon_1 \in (0,1)$ such that, if $\epsilon < \epsilon_1$, inequalities (\ref{eq59}) and (\ref{eq63}) can be rewritten as:
\begin{align*}
& \norm{\tilde{T}^{p_{m-2} p_{m-1}}\charf{F(\Omega,t_{m-1})}\ldots \tilde{T}^{0 p_{0}}\charf{F(\Omega,0)}\mu}\nonumber\\
& \qquad\qquad\leq K \theta_0 \left[\left(2\  \left(\tfrac{D_0}{1- \lambda_0}\right) \ \left(\lambda_0 +\tfrac{D_0}{1-\lambda_0}\right)\right)\, \sigma_{1}\right]^{\frac{m}{2}} \leq K \theta_0\, \frac{\sigma^t}{(4 \gamma)^{m}}\\
& \norm{\tilde{T}^{p_{m-2} p_{m-1}}\charf{F(\Omega,t_{m-1})}\ldots \tilde{T}^{0 p_{0}}\charf{F(\Omega,0)}\mu}\nonumber\\
&\qquad\qquad\leq \frac{K}{26 (1 - 81 \alpha_0)} \; (18\,\sqrt{\alpha_0})^{m} \leq \frac{K}{26 (1 - 81 \alpha_0)} \,\frac{\sigma^t}{(4 \gamma)^m}
\end{align*}
Therefore, for some constant $C$, we have, for any $\Omega$:
\begin{equation*}
\norm{\tilde{T}^{p_{m-1} p_{m}}\charf{F(\Omega,t_{m-1})} \ldots \tilde{T}^{p_0 p_1}\charf{F(\Omega,0)} \proj{0 p_{0}}\mu} \leq C\ (4 \gamma)^{-m} \ \sigma^t .
\end{equation*}
With this estimate, we control all the terms of in the sum of (\ref{eq58}), and since the number of terms in the sum is controlled by (\ref{eq58.1}), we get:
\begin{equation}\label{eq63.1}
\norm{T^t \mu (\varphi)} \leq \norm{\Lambda}\,(4 \gamma)^{\frac{t}{\gamma}}\ C\ K\ (4 \gamma)^{- \frac{t}{\gamma}} \ \sigma^t \leq C\  \norm{\Lambda} \sigma^t \,\norm{\varphi}_{\infty}.
\end{equation}

Finally, if $t$ is not a multiple of $\gamma$, we know that $t$ can be rewritten as $t = m \gamma + t'$, for some $m \in \nat$ and $t' < \gamma$. Then, we can apply (\ref{eq63.1}) to $\mu$ and $\varphi \circ T^{t'}$, which depends only on the variables in a set of size at most $t' \norm{\Lambda}$, and we find:
\begin{equation}
\norm{T^t \mu (\varphi)} \leq  C \ t'\,\norm{\Lambda} \sigma^{m\gamma} \,\norm{\varphi}_{\infty} \leq \frac{\gamma C }{\sigma^{\gamma}} \,\norm{\Lambda}\, \sigma^t\,\norm{\varphi}_{\infty}
\end{equation}
which is exactly the promised result, up to a redefinition of the constant $C$.\end{proof}

Theorem \ref{result2} is now a trivial consequence of Theorem \ref{prop12}. 

\begin{proof}[Theorem \ref{result2}]
We already know from Theorem \ref{result1} that $\muinv{-}$ belongs to $\mathcal{B}(K_0,3 \alpha_0,\theta_0)$. And since $\mu$ belongs to $\mathcal{B}(K,3 \alpha_0,\theta_0)$.
\begin{equation*}
\tnorm{\charf{(0,1]^{\Lambda}}(\mu - \muinv{-})}{\theta_0} \leq K \, (3 \alpha_0)^{\norm{\Lambda}} \norm{\mu} + K_0 {(3 \alpha_0)}^{\norm{\Lambda}} \norm{\muinv{-})} \leq \left(K + K_0\right)\,(3 \alpha_0)^{\norm{\Lambda}}.
\end{equation*}
So, for $K' = K + K_0$, $(\mu - \muinv{-})$ is a signed measure of zero mass in $\mathcal{B}(K',3\alpha_0,\theta_0)$, and Theorem \ref{prop12} implies that:
\begin{equation*}
\norm{\mu(\varphi \circ T) - \muinv{-}(\varphi)} \leq C\ \norm{\Lambda}\ \sigma^t \ \norm{\varphi}_{\infty} .
\end{equation*}
\end{proof}

\section{Exponential Decay of Correlations}
\label{expdecay}

Theorem \ref{result2} implies that the spatial correlations of $\muinv{-}$ decay exponentially:
\begin{proof}[Proposition \ref{result3}]
Since $\sleb{-}$ belongs to $\mathcal{B}(1,3 \alpha_0,\theta_0)$, we know that for any continuous function $\varphi$ depending only on the variables in $\Sigma$:
\begin{equation}\label{eq65}
\norm{\sleb{-}( \varphi \circ T^t) - \muinv{-}(\varphi)} \leq C'\ \norm{\Sigma}\ \sigma^t \norm{\varphi}_{\infty}.
\end{equation}

If we choose $t = d(\Lambda,\Omega) - 1$, $\varphi \circ T^t$ and $\psi \circ T^t$ depends on different variables, and since $\sleb{-}$ is a totally decoupled, we have $\sleb{-}( \varphi \circ T^t\, \psi \circ T^t  ) = \sleb{-}(\varphi \circ T^t )\,\sleb{-}( \psi \circ T^t )$. Therefore:
\begin{align*}
& \norm{\,\muinv{-}( \varphi \psi ) - \muinv{-}(\varphi)\, \muinv{-}(\psi)\,}\\
& \qquad  \leq \norm{\,\muinv{-}( \varphi \psi ) - T^t \sleb{-}(\varphi \psi)\,} + \norm{\,T^t \sleb{-}(\varphi \psi) - \muinv{-}(\varphi)\, T^t\sleb{-}(\psi)\,}\\
& \qquad \qquad + \norm{\,\muinv{-}(\varphi)\, T^t\sleb{-}(\psi) -   \muinv{-}(\varphi)\, \muinv{-}(\psi)\,}\\
& \qquad \leq  \norm{\,\muinv{-}( \varphi \psi ) - T^t \sleb{-}(\varphi \psi)\,} + \norm{\,T^t \sleb{-}(\varphi) - \muinv{-}(\varphi)\,} \norm{\,T^t\sleb{-}(\psi)\,} \\
& \qquad \qquad + \norm{\,\muinv{-}(\varphi) \,T^t\sleb{-}(\psi) -   \muinv{-}(\varphi) \,\muinv{-}(\psi)\,}.
\end{align*}
This inequality, together with (\ref{eq65}), yields:
\begin{align*}
& \norm{\,\muinv{-}( \varphi \psi ) - \muinv{-}(\varphi) \,\muinv{-}(\psi)\,}\\
& \qquad  \leq  C'\ \norm{\Lambda \cup \Omega} \sigma^t \norm{\varphi}_{\infty} \norm{\psi}_{\infty}+ C'\ \norm{\Lambda} \sigma^t \norm{\varphi}_{\infty} \norm{\psi}_{\infty} + C'\ \norm{\Omega} \sigma^t \norm{\varphi}_{\infty} \norm{\psi}_{\infty}\\
&\qquad \leq  2 \ C'\ \norm{\Lambda \cup \Omega} \sigma^t \norm{\varphi}_{\infty} \norm{\psi}_{\infty}.
\end{align*}
If we take $C = 2 C'$, we have the exponential decay of correlation in space for the invariant measure $\muinv{-}$.
\qquad$\qed$\end{proof}

The proof of the exponential decay in time also follows from standard arguments, but we have to put some additional regularity assumptions on $\psi$ because $\psi \muinv{-}$ has to belong to $\mathcal{B}(K,3\alpha_0,\theta_0)$.

\begin{proof}[Proposition \ref{result4}]
Without loss of generality, assume that $\norm{\varphi}_{\infty} \leq 1$ and $\norm{\psi}_{\infty} \leq 1$. Consider then the measure $\mu_{\psi}$ defined by:
\begin{equation*}
\mu_{\psi}(\varphi) \equiv \muinv{-}(\varphi\,\psi) - \muinv{-}(\varphi)\, \muinv{-}(\psi).
\end{equation*}
In order to prove the exponential decay in time, we just have to prove that $\mu_{\psi}$ satisfies the assumptions of Theorem \ref{prop12}. It is easy to check that $\mu_{\psi}$ is a signed measure of zero mass. And since $\muinv{-}(\varphi)\, \muinv{-}$ already belongs to $\mathcal{B}(K_0,3\alpha_0,\theta_0)$, we just have to prove that $\psi\muinv{-}$ also belongs to $\mathcal{B}(K,3\alpha_0,\theta_0)$ for some $K > 0$.

Since $\psi$ depends only on the variables inside some finite set $\Lambda \subset \ent$ and belongs to $\mathcal{C}_{\Lambda}^1(X)$, we have:
\begin{equation*}
\psi \partial_{\tilde{\Lambda}} \tilde{\varphi} = \sum_{V \subseteq \tilde{\Lambda}}\,(-1)^{\norm{\tilde{\Lambda} \setminus V}}\, \partial_V \left( \tilde{\varphi} \partial_{\tilde{\Lambda} \setminus V} \psi\right),
\end{equation*}
for any $\tilde{\Lambda} \subseteq \ent$ and any $ \tilde{\varphi} \in \mathcal{C}_{\tilde{\Lambda}}^1(X)$. Therefore:
\begin{align*}
 \psi\,\partial_{\tilde{\Lambda}}  \tilde{\varphi} & = \partial_{\tilde{\Lambda} \setminus \Lambda} \left( \psi\,\partial_{\tilde{\Lambda} \cap \Lambda}  \tilde{\varphi}\right)\\
 & = \sum_{V \subseteq \tilde{\Lambda} \cap \Lambda} (-1)^{\norm{(\tilde{\Lambda} \cap \Lambda) \setminus V}}\, \partial_{V\cup (\tilde{\Lambda}\setminus\Lambda )}\left(\tilde{\varphi}\, \partial_{(\tilde{\Lambda} \setminus \Lambda) \setminus V} \psi\right).
 \end{align*}
This implies that:
\begin{align*}
\var{\tilde{\Lambda}}\left(\charf{(0,1]^{\Omega}}\psi \muinv{-}\right) & \leq \sum_{V \subseteq \tilde{\Lambda} \cap \Lambda}  \var{V\cup (\tilde{\Lambda} \setminus \Lambda)}\left(\charf{(0,1]^{\Omega}}\muinv{-}\right)\ \norm{\partial_{(\tilde{\Lambda} \cap \Lambda) \setminus V} \psi}_{\infty}\\
& \leq K_0 (3 \alpha_0)^{\norm{\Omega}}\ \sum_{V \subseteq \tilde{\Lambda} \cap \Lambda}  {\theta_0}^{\norm{V\cup (\tilde{\Lambda} \setminus \Lambda)}} \ \norm{\partial_{(\tilde{\Lambda} \cap \Lambda) \setminus V} \psi}_{\infty}
\end{align*}
Since $\psi$ belongs to $\mathcal{C}_{\Lambda}^1(X)$ and since $I^{\Lambda}$ is a compact set, there is some constant $C'$ such that $\norm{\partial_{(\tilde{\Lambda} \cap \Lambda) \setminus V} \psi}_{\infty} \leq C'$ for any set $V$. Therefore: 
\begin{align*}
\var{\tilde{\Lambda}}\left(\charf{(0,1]^{\Omega}}\psi \muinv{-}\right) & \leq K_0\,C'\, (3 \alpha_0)^{\norm{\Omega}}\ \sum_{V \subseteq \tilde{\Lambda} \cap \Lambda}  {\theta_0}^{\norm{V\cup (\tilde{\Lambda} \setminus \Lambda)}}\\
& \leq K_0\,C'\, (3 \alpha_0)^{\norm{\Omega}}\ \sum_{V \subseteq \Lambda}  {\theta_0}^{\norm{V} + \norm{\tilde{\Lambda}} - \norm{\Lambda}}
\end{align*}
Multiplying each side of the inequality by $\theta_0^{- \norm{\tilde{\Lambda}}}$ and taking the supremum over all finite $\tilde{\Lambda} \subset \ent$ yields:
\begin{align*}
\tnorm{\charf{(0,1]^{\Omega}}\psi \muinv{-}}{\theta_0}& \leq K_0 C' (3 \alpha_0)^{\norm{\Omega}}\ \sum_{V \subseteq \Lambda}  {\theta_0}^{\norm{V}- \norm{\Lambda}}\\
& \leq K_0 C'\left(\frac{1+ \theta_0}{\theta_0}\right)^{\norm{\Lambda}} \  (3 \alpha_0)^{\norm{\Omega}}
\end{align*}
So, $\psi\muinv{-}$ belongs to $\mathcal{B}(K,3\alpha_0,\theta_0)$ for some $K > 0$. Therefore, $\mu_{\psi}$ also belongs to $\mathcal{B}(K,3\alpha_0,\theta_0)$ for another $K > 0$, and we can apply Theorem \ref{prop12} to the signed measure of zero mass $\mu_{\psi}$. This yields:
\begin{equation*}
\norm{T^t\mu_{\psi}(\varphi)} = \norm{\muinv{-}(\varphi\circ T^t\,\psi) - \muinv{-}(\varphi)\, \muinv{-}(\psi)} \leq C_{\varphi,\psi} \norm{\Lambda}\ \sigma^t.
\end{equation*}
\end{proof}

\section*{Acknowledgments}
The author would like to thank Jean Bricmont, Carlangelo Liverani and Christian Maes for useful comments and discussions.
\appendix
\section{Regularization estimates}
We reported to this appendix all the regularization estimates needed in this article.

 Let $\Lambda$ be some arbitrary finite subset of $\ent$. Let $\eta : I^{\Lambda} \mapsto \real$ be a non-negative, real-valued function in $\mathcal{C}_{0}^{\infty}(\real^{\Lambda})$, the space of continuous functions on $\real^{\Lambda}$ with compact support. Then, if $\varphi: I^{\Lambda} \mapsto I^{\Lambda}$ is integrable, one can define the regularization of $\varphi$ for every $\epsilon >0$,
\begin{equation}\label{eqA1}
\big(\eta_{\epsilon}\star \varphi\big)(\vect{x}) = \int_{I^{\Lambda}} \eta(\vect{w})\, \varphi(\vect{x} + \epsilon \vect{w}) \,\dif{\vect{w}} = \epsilon^{- \norm{\Lambda}}\;\int_{I^{\Lambda}} \eta\left(\frac{\vect{x} - \vect{w}}{\epsilon}\right)\, \varphi(\vect{w}) \,\dif{\vect{w}}
\end{equation}
where we assumed $\varphi$ to be identically zero outside $I^{\Lambda}$. This regularization of $\varphi$ has the following properties:
\begin{proposition}\label{propA1}
For all $\varphi \in \mathcal{C}^0 (I^{\Lambda}) $, we have
\begin{enumerate} 
\item  $\norm{\eta_{\epsilon} \star \varphi}_{\infty} \leq \norm{\varphi}_{\infty}$
\item ${\displaystyle\lim_{\epsilon \to 0}\; \norm{ (\eta_{\epsilon}\star \varphi)  - \varphi}_{\infty} = 0}$
\item Moreover, if $\varphi \in \mathcal{C}^1(I^{\Lambda})$, then $\partial_p (\eta_{\epsilon}\star\varphi) = (\eta_{\epsilon}\star \partial_p \varphi) $ 
\end{enumerate}
\end{proposition}

The proof of these results can be found in textbooks (see for instance \cite{Gi84} or \cite{Zi89}). However, the functions we are interested in are usually not continuous but only piecewise continuous in the following sense: we say that $\varphi$ is piecewise continuous on $I^{\Lambda}$ if there is a finite family of open intervals $I_{i}$ such that $\bigcup_{i} \overline{I}_i  = I$ and $\varphi$ is continuous on $I_{i_1} \times \ldots \times I_{i_{\norm{\Lambda}}}$ for every $\{i_{p}\}_{p \in \Lambda}$.

The following proposition shows that the piecewise continuity of the function $\varphi$ will not play an important role if $\varphi$ is integrated with respect to a measure in $L^1(X)$:
\begin{proposition}\label{propA2}
Let $\mu \in L^1(X)$ and $\varphi$ be a piecewise continuous on $I^{\Lambda}$. Then:
\begin{equation*}
{\displaystyle \lim_{\epsilon \to 0} \mu (\,\eta_{\epsilon} \star \varphi\,) = \mu (\,\varphi\,)}
 \end{equation*}
\end{proposition}
\begin{proof}
Take some arbitrary $\alpha > 0$. From (\ref{eqA1}), we know that:
\begin{equation}\label{eqA2} \norm{\mu (\,\eta_{\epsilon} \star \varphi\, -  \,\varphi\,) }  \leq \int_{I^{\Lambda}} \mu(\dif{\vect{x}}) \:\int \dif{\vect{w}} \; \norm{\varphi(\vect{x} + \epsilon \vect{w}) - \varphi(\vect{x})} \;\eta(\vect{w})\end{equation}
Let $B_{\epsilon}$ be a strip of size $\epsilon$ around the discontinuities of $\varphi$. If $\vect{x}$ does not belong to $B_{\epsilon}$, $\vect{x} + \epsilon \vect{w}$ and $\vect{x}$ are in the same $I_{i_1} \times \ldots \times I_{i_{\norm{\Lambda}}}$, and the piecewise continuity of $\varphi$ implies that, if $\epsilon$ is small enough,  $\norm{\varphi(\vect{x} + \epsilon \vect{w}) - \varphi(\vect{x})} \leq \alpha$. If $\vect{x}$ belongs to $B_{\epsilon}$, we have the following trivial bound: $\norm{\varphi(\vect{x} + \epsilon \vect{w}) - \varphi(\vect{x})} \leq 2 \,\norm{\varphi}_{\infty}$. So, (\ref{eqA2}) becomes: 
\begin{equation*} \norm{\mu (\,\eta_{\epsilon} \star \varphi\, -  \,\varphi\,) }  \leq \mu(B_{\epsilon}^{c}) \, \alpha + 2  \mu(B_{\epsilon}) \,\norm{\varphi}_{\infty}  \leq \mu(1) \, \alpha + 2  \mu(B_{\epsilon}) \,\norm{\varphi}_{\infty}  \end{equation*}
When $\epsilon$ goes to $0$, both terms of the sum vanishes. This is trivial for the first one, and for the second one, we just have to notice that when $\epsilon$ goes to zero, $B_{\epsilon}$ shrinks to a set of zero Lebesgue measure in $I^\Lambda$ and that $\mu$ on $I^\Lambda$ is absolutely continuous with respect to the Lebesgue measure. 
\qquad$\qed$\end{proof}

However, we are not only interested in the $L^1(X)$ norm but also in the variations semi-norms, such as $\var{\Lambda'}$. The following Lemma shows that in some sense, the regularizations and the derivatives do commute:
\begin{lemma}\label{propA3}
Let $\Lambda'$ be some finite subset of $\ent$, and $\varphi: I^\Lambda \to \real$ such that:
\begin{itemize}
\item $\varphi$ is piecewise continuous with respect to $x_p$ for $p \notin \Lambda'$
\item $\varphi$ is continuous with respect to $x_p$ for $p \in \Lambda'$
\item $\varphi$ is piecewise continuously differentiable with respect to $x_p$ for $p \in \Lambda'$
\end{itemize}
Then, for every $\mu$ such that $\var{\Lambda'} \mu < \infty$, we have:
\begin{equation*}
\lim_{\epsilon \to 0} \mu\left(\eta_{\epsilon} \star \partial_{\Lambda'} \varphi \right)= \lim_{\epsilon \to 0} \mu\left( \partial_{\Lambda'} \left(\eta_{\epsilon}\star\varphi\right)\right)
\end{equation*}
\end{lemma}
\begin{proof}
The proof of the result is a simple recurrence on the size of $\Lambda'$. If $\Lambda' = \{p\}$, the continuity of $\varphi$ with respect to $x_p$ implies that:
\begin{equation*}
\varphi(\vect{x}) = \varphi(0,\vect{x}_{\neq p}) + \int_{0}^{x_p} \partial_p \varphi(\xi_p,\vect{x}_{\neq p}) \, \dif \xi_p.
\end{equation*}
Let us then define $\varphi_{\epsilon}$:
\begin{equation*}
\varphi_{\epsilon}(\vect{x}) = \eta_{\epsilon} \star \varphi(0,\vect{x}_{\neq p}) + \int_{0}^{x_p} \eta_{\epsilon} \star \partial_p \varphi(\xi_p,\vect{x}_{\neq p}) \, \dif \xi_p.
\end{equation*}
Then, $\partial_p \varphi_{\epsilon}(\vect{x})= \eta_{\epsilon} \star \partial_p \varphi$, and:
\begin{align*}
\lim_{\epsilon \to 0} \mu\left(\eta_{\epsilon} \star \partial_p \varphi - \partial_p \left(\eta_{\epsilon}\star\varphi\right)\right) & = \lim_{\epsilon \to 0} \mu\left(\partial_p \varphi_{\epsilon}  - \partial_p \left(\eta_{\epsilon}\star\varphi\right)\right)\\
& \leq \dnorm{\mu} \lim_{\epsilon \to 0} 
\norm{\varphi_{\epsilon} - \eta_{\epsilon}\star\varphi}_{\infty} .
\end{align*}
But now, since both $\varphi_{\epsilon}$ and $\eta_{\epsilon}\star\varphi$ tends towards $\varphi$ in the sup norm when $\epsilon$ goes to zero, this complete the proof for $\Lambda' = \{p\}$.

Take now some arbitrary $\Lambda'$, and assume that for some $ q \in \Lambda'$, the property is true in $\Lambda'\setminus \{q\}$. Therefore:
\begin{align*}
\lim_{\epsilon \to 0} \mu\left(\eta_{\epsilon} \star \partial_{\Lambda'}  \varphi - \partial_{\Lambda'}  \left(\eta_{\epsilon}\star\varphi\right)\right)= \lim_{\epsilon \to 0} \mu\left(\partial_{\Lambda'\setminus \{q\}}\left(\eta_{\epsilon} \star  \partial_q \varphi - \partial_q \left(\eta_{\epsilon}\star\varphi\right)\right)\right).
\end{align*}
Let's then define 
\begin{equation*}
\varphi_{\epsilon}(\vect{x}) = \eta_{\epsilon} \star \varphi(0,\vect{x}_{\neq q}) + \int_{0}^{x_q} \eta_{\epsilon} \star \partial_q \varphi(\xi_q,\vect{x}_{\neq q}) \, \dif \xi_q.
\end{equation*}
Then, once again, $\partial_q \varphi_{\epsilon}= \eta_{\epsilon} \star \partial_q \varphi$. So:
\begin{equation*}
\lim_{\epsilon \to 0} \mu\left(\partial_{\Lambda'\setminus \{q\}}\left(\eta_{\epsilon} \star  \partial_q \varphi - \partial_q \left(\eta_{\epsilon}\star\varphi\right)\right)\right) \leq \var{\Lambda'}\mu \, \lim_{\epsilon \to 0} \norm{\varphi_{\epsilon} - \eta_{\epsilon}\star\varphi}_{\infty} = 0.
\end{equation*}
\end{proof}

\begin{proposition}\label{propA4}
For every finite $\Lambda \subset \ent$, and $\varphi: X \to \real$ continuous and piecewise continuously differentiable with respect to every $x_p$ with $p \in \Lambda$, and piecewise continuous with respect to the other variables $\vect{x}_{\neq \Lambda }$. Then: 
\begin{equation*}
\mu\left(\partial_{\Lambda} \varphi\right) \leq \var{\Lambda}(\mu) \, \norm{\varphi}_{\infty}
\end{equation*}
\end{proposition}

\begin{proof}
Since $\partial_{\Lambda} \varphi$ is piecewise continuous, if $\eta_{\epsilon}$ is some positive symmetric mollifier for $\varphi$, we have:
\begin{equation}\label{eqA3}\begin{split}
& \mu\left(\partial_{\Lambda} \varphi\right) = \lim_{\epsilon \to 0}\mu\left(\eta_{\epsilon} \star \partial_{\Lambda} \varphi\right)
\end{split}\end{equation}
But then, using Lemma \ref{propA3}, equation (\ref{eqA3}) become:
\begin{equation*}\begin{split} 
& \mu\left(\partial_{\Lambda} \partial_p \varphi\right) =  \lim_{\epsilon \to 0}\mu\left(\partial_{\Lambda} \,\eta_{\epsilon} \star   \varphi\right)\\
& \qquad \leq \lim_{\epsilon \to 0}\var{\Lambda}(\mu) \norm{\eta_{\epsilon} \star  \varphi}_{\infty} \leq  \var{\Lambda}(\mu) \norm{\varphi}_{\infty}
\end{split}\end{equation*}
\end{proof}
And, as a direct corollary, we have:
\begin{corollary}\label{propA5}
For any interval $[a,b] \subseteq [-1,1]$, every $p$ and every $\Lambda$, we have:
\begin{equation*}
\var{\Lambda}(\charf{[a,b]}(x_p) \mu) \leq \var{\Lambda}(\mu)
\end{equation*}

\end{corollary}
\begin{proof}
If $p \notin \Lambda$, the result is a direct consequence of Proposition \ref{propA4}. For any function $\varphi$ in $\mathcal{C}_{\Lambda}^1(X)$, $p \notin \Lambda$ implies that $\charf{[a,b]}(x_p) \partial_{\Lambda} \varphi = \partial_{\Lambda} \charf{[a,b]}(x_p)  \varphi $, and since $\charf{[a,b]}(x_p) \varphi$ is continuously differentiable with respect to the variables in $\Lambda$, and piecewise continuous with respect to the other variables, we have:
\begin{equation*}
\mu( \charf{[a,b]}(x_p) \partial_{\Lambda} \varphi) = \mu( \partial_{\Lambda} \charf{[a,b]}(x_p) \varphi)  \leq \var{\Lambda}(\mu) \ \norm{\varphi}_{\infty}
\end{equation*}

If $p \in \Lambda$, for any function $\varphi \in \mathcal{C}_{\Lambda}^1(X)$, we can define:
\begin{equation*}
\psi(\vect{x}) = \varphi(a,\vect{x}_{\neq p}) + \int_{a}^{x_p} \charf{[a,b]}(\xi_p)\  \partial_p \varphi (\xi_p,\vect{x}_{\neq p})\  \mathrm d \xi_p
\end{equation*}
We immediately see that $\psi$ is continuous, piecewise continuously differentiable, and $\partial_p \psi =  \charf{[a,b]}(\xi_p) \partial_p \varphi$. Therefore:
\begin{equation*}
\mu(\charf{[a,b]}(\xi_p) \partial_{\Lambda} \varphi) = \mu(\partial_p \psi ) \leq \var{\Lambda}(\mu) \norm{\psi}_{\infty}
\end{equation*}
But the definition of $\psi$ implies that $\norm{\psi}_{\infty} \leq  \norm{\varphi}_{\infty}$, and we therefore have:
\begin{equation*}
{\mu(\charf{[a,b]}(\xi_p) \partial_{\Lambda} \varphi) \leq \var{\Lambda}(\mu) \norm{\varphi}_{\infty}} 
\end{equation*}
\end{proof}

\bibliography{biblio}{}
\bibliographystyle{unsrt}
\end{document}